\documentclass[12pt,reqno]{amsart}
\usepackage{fullpage}
\usepackage{amsfonts}
\usepackage{amssymb}
\usepackage{enumerate}
\usepackage{times}
\usepackage{graphicx}
\usepackage{url}

\newtheorem{theorem}{Theorem}[section]
\newtheorem{prop}[theorem]{Proposition}

\newtheorem{lemma}[theorem]{Lemma}
\newtheorem*{conj}{Conjecture}
\newtheorem{subconj}{Conjecture}
{
\theoremstyle{remark}

}

\renewcommand{\mod}[1]{{\ifmmode\text{\rm\ (mod~$#1$)}\else\discretionary{}{}{\hbox{ }}\rm(mod~$#1$)\fi}}
\newcommand{\ep}{\varepsilon}

\newcommand{\C}{{\mathbb C}}
\renewcommand{\L}{{\mathcal L}}
\newcommand{\N}{{\mathbb N}}

\newcommand{\R}{{\mathbb R}}
\renewcommand{\S}{{\mathcal S}}
\newcommand{\Z}{{\mathbb Z}}

\begin{document}

\title[Nonzero values of Dirichlet $L$-functions in vertical arithmetic progressions]{Nonzero values of Dirichlet $L$-functions in vertical arithmetic progressions} 

\author[Greg Martin]{Greg Martin}
\address{Department of Mathematics \\ University of British Columbia \\ Room 121, 1984 Mathematics Road \\ Vancouver, BC, Canada \ V6T 1Z2}
\email{gerg@math.ubc.ca}
\author[Nathan Ng]{Nathan Ng}
\address{University of Lethbridge \\ Department of Mathematics and Computer Science \\ 4401 University Drive \\ Lethbridge, AB, Canada \ T1K 3M4 }
\email{nathan.ng@uleth.ca}


\begin{abstract}
Let $L(s,\chi)$ be a fixed Dirichlet $L$-function. Given a vertical arithmetic progression of $T$ points on the line $\Re s=\frac12$, we show that at least $cT/\log T$ of them are not zeros of $L(s,\chi)$ (for some positive constant~$c$). This result provides some theoretical evidence towards the conjecture that all ordinates of zeros of Dirichlet $L$-functions are linearly independent over the rationals. We also establish an upper bound (depending upon the progression) for the first member of the arithmetic progression that is not a zero of $L(s,\chi)$.
\end{abstract}

\maketitle

\section{Introduction}

Every Dirichlet $L$-function $L(s,\chi)$ has infinitely many nontrivial zeros, and our knowledge about their vertical distribution is rather good. Our understanding of the algebraic nature of their imaginary parts, on the other hand, is extremely meager. We have no reason to doubt that these imaginary parts are haphazard transcendental numbers that have no algebraic relationships among them; in particular, any vertical arithmetic progression should contain at most two zeros of $L(s,\chi)$. We are very far, however, from being able to prove such a statement. The purpose of this article is to improve our knowledge concerning the number of zeros of $L(s,\chi)$ in a vertical arithmetic progression.

Our main theorem provides a lower bound for the number of points in an arithmetic progression on the critical line at which a Dirichlet $L$-function takes nonzero values.

\begin{theorem} \label{main theorem}
Let $\chi$ be a primitive Dirichlet character modulo $q$ where $q \ge 2$, and let $a$ and $b$ be real numbers with $b\ne0$. Then  
\begin{equation}
     \#  \big\{ 1 \le k  \le T \colon L\big(\tfrac12+i (a+k b), \chi \big) \ne 0  \big\}   \gg  \frac{T}{\log T},
\end{equation}
where the implicit constant depends on $a$, $b$, and~$q$.
\end{theorem}

\noindent This result holds for all non-principal Dirichlet characters, as each character is induced by a primitive character.
As stated, the theorem does not cover the case of the Riemann zeta function; however, a minor modification of our methods would establish the analogous lower bound $\#  \big\{ 1 \le k  \le T \colon \zeta\big(\tfrac12+i (a+k b) \big) \ne 0  \big\}   \gg  {T}/{\log T}$.  

We remark that our method of proof would work perfectly well for other vertical lines $\Re(s) = \sigma$ inside the critical strip besides the critical line $\Re(s) = \frac12$ itself. On these other lines, however, a stronger result is already implied by zero-density estimates. Specifically, Linnik~\cite{Li} proved that for any $0<\delta<\frac12$, there exists $0<\ep(\delta)<1$ such that $L(s,\chi)$ has $\ll_{\chi,\delta} T^{1-\ep(\delta)}$ zeros in the rectangle $\{ \frac12+\delta \le \Re(s) \le 1$, $|\Im(s)|\le T\}$. By the functional equation for Dirichlet $L$-functions, the same quantity bounds the number of zeros in the rectangle in the rectangle $\{ 0 < \Re(s) \le \frac12-\delta$, $|\Im(s)|\le T\}$. In particular, $L(s,\chi)$ is nonvanishing at almost all the points in any vertical arithmetic progression $\sigma+i(a+kb)$ with $\sigma\ne\frac12$. (The lone exception is when $\chi$ is imprimitive, in which case $L(s,\chi)$ does have infinite arithmetic progressions of zeros on the imaginary axis; but these are easily understood.)

Putnam \cite{Pu1} proved that the set of zeros of $\zeta(s)$ does not contain any sequence of the form $\{ \tfrac12+i k b\colon k\in\N\}$. (In fact, this result immediately implies that such a sequence contains infinitely many points that are not zeros of $\zeta(s)$, via the simple trick of replacing $b$ by a multiple of itself.) Later, Lapidus and van Frankenhuysen~\cite[Chapter 9]{LvF} extended this result to a certain class of $L$-functions including Dirichlet $L$-functions. Moreover, they showed that there exist arbitrarily large values of $T$ such that 
\begin{equation*}
     \#  \big\{ 1 \le k  \le T \colon L\big(\tfrac12+i k b, \chi \big) \ne 0  \big\}  \gg T^{5/6}.
\end{equation*}
Theorem~\ref{main theorem} both improves this lower bound and generalizes it to inhomogeneous arithmetic progressions $\tfrac12+i (a+k b)$ (as Lapidus and van Frankenhuysen's result applies only to the case $a=0$).

We remark that stronger results are available if one is willing to assume the generalized Riemann hypothesis for Dirichlet $L$-functions. Under this hypothesis, Ford, Soundararajan, and Zaharescu~\cite[Theorem 6(ii)]{FSZ} establish for any $\ep>0$ the inequality
\begin{equation} \label{FSZ inequality}
\bigg| \sum_{\substack{0<\gamma\le T \\ \{\alpha\gamma\} \in \mathbb I}} 1 - |\mathbb I| N(T,\chi) - T \int_{\mathbb I} g_{\chi,\alpha}(x) \,dx \bigg| \le \alpha(1/2+\ep)T + o(T)
\end{equation}
involving the fractional-part notation $\{ x \} = x - \lfloor x \rfloor$; the sum is over all imaginary parts $\gamma$ of zeros of $L(s,\chi)$ such that the scaled fractional part $\{\alpha x\}$ lies in a given subinterval $\mathbb I$ of $\R/\Z$. (For our purposes the precise definitions of $N(T,\chi)$ and $g_{\chi,\alpha}(x)$ are irrelevant; all we need is that they and the error term $o(T)$ are bounded uniformly in the interval $\mathbb I$ for fixed $\chi$, $T$, and~$\alpha$.) We set $\alpha=1/b$, replace $T$ by $bT$, and let $\mathbb I$ be an interval centered at $a/b$ of vanishingly small diameter. In the limit as the length of $\mathbb I$ goes to zero, the second and third terms on the left-hand side of equation~\eqref{FSZ inequality} vanish, and we obtain
\[
\bigg| \sum_{\substack{0<\gamma\le bT \\ \{\gamma/b\} = \{a/b\}}} 1 \bigg| \le \frac1b(1/2+\ep)bT + o(T).
\]
On noting that the sum counts precisely those imaginary parts $\gamma$ such that $(\gamma-a)/b$ is an integer, we see that the left-hand side is precisely equal (since we are assuming GRH) to
\begin{multline*}
\#  \big\{ {-}\tfrac ab < k  \le T-\tfrac ab \colon L\big(\tfrac12+i (a+k b), \chi \big) = 0 \big\} \\
= \#  \big\{ 1 \le k  \le T \colon L\big(\tfrac12+i (a+k b), \chi \big) = 0 \big\} + O(1).
\end{multline*}
Since $\ep>0$ can be chosen arbitrarily, we conclude that at least $(1/2+o(1))T$ of the first $T$ points in the arithmetic progression $\tfrac12+i (a+k b)$ are not zeros of $L(s,\chi)$. If one is willing to make stronger assumptions (a version of the pair correlation conjecture, or a statement about the distribution of primes in short intervals, for example), then corresponding results of \cite{FSZ} show that all but $o(T)$ of the first $T$ points in the arithmetic progression $\tfrac12+i (a+k b)$ are not zeros of $L(s,\chi)$.

The technique used to prove Theorem~\ref{main theorem} can be adapted to bound the least term in an arithmetic progression on the critical line that is not a zero of $L(s,\chi)$.

\begin{theorem}  \label{lowest nonzero theorem}
Let $a$ and $b$ be real numbers satisfying $0\le a < b$, let $D > 8 \log 2$ be a real number, and let $\chi$ be a non-principal Dirichlet character\mod q.
\begin{enumerate}
\item If $b \ge 2 \pi$, then there exists a positive integer 
\[
k \ll_{D}  
     b^3 q \exp\Big( \frac{D \log(b^3 q)}{\log\log(b^3 q)} \Big)
\] 
for which $L(\frac12+i (a+k b),\chi) \ne 0$.
\item If $0 < b \le 2 \pi$, then there exists a positive integer 
\begin{equation*}
k \ll_{D}   b^{-1} q \exp \Big( \frac{D \log q}{\log\log q} \Big)
\end{equation*}
for which $L(\frac12+i (a+k b),\chi) \ne 0$.
\end{enumerate}
\end{theorem}

Again, a minor modification of our methods would establish, for $b\ge2\pi$, that there exists a positive integer $k \ll_{D} b^3 \exp\big( {(D \log b^3)}/{\log\log b^3} \big)$ for which $\zeta\big(\frac12+i(a+kb)\big) \ne0$. In the case $a=0$, van Frankenhuijsen~\cite{vF} established that $\zeta(\tfrac12+ikb)$ is nonzero for some positive integer $k\le 13b$, which is superior in this homogeneous case. (Note that when $0<b<2\pi$, the very first term $\frac12+i (a+ b)$ in the arithmetic progression is not a zero of $\zeta(s)$, since the lowest zero of the Riemann zeta-function has imaginary part exceeding~14.) In an unpublished manuscript, Watkins~\cite{Wa1} showed that $L(\frac12+i k b, \chi) \ne 0$ for some positive integer
\begin{equation}
  \label{eq:bound for k}
k \ll_\chi \exp\big( {(\tfrac{8}{3 \pi}+o(1)) b (\log b)^2} \big).
\end{equation}
Theorem~\ref {lowest nonzero theorem} results in a better dependence upon $b$; we also make explicit the dependence upon $q$ and allow inhomogeneous arithmetic progressions $\frac12+i (a+k b)$.

Motivation for considering the above theorems is provided by the following far-reaching conjecture:

\begin{conj}[{\bf ``LI''}] \label{LI}
 The nonnegative imaginary parts of the nontrivial zeros of Dirichlet $L$-functions corresponding
to primitive characters are linearly independent over the rationals. 
\end{conj}

It appears Hooley \cite{Ho} was the first to mention this conjecture in the context of Dirichlet $L$-functions. 
Recently there has been significant interest in this  conjecture, as
Rubinstein and Sarnak  \cite{RS} established theorems revealing a close connection between LI and ``prime number races'' in comparative prime number theory. 
Following their work, several researchers investigated the link between
LI and prime number races (see the survey of Granville and the first author \cite{GM} for references to the literature).

The LI conjecture is a natural generalization of a much older conjecture for the zeros of the Riemann zeta function:

\begin{subconj}  \label{LI zeta}
The positive imaginary parts of the nontrivial zeros of the Riemann zeta function
are linearly independent over the rationals. 
\end{subconj} 

Wintner~\cite{Wi2a, Wi1, Wi2} mentions the connection between the diophantine properties of the imaginary parts of
the zeros of the zeta function and the asymptotic distribution of the remainder term in the prime number theorem.
In \cite{Wi2} he indicates that Conjecture~\ref{LI zeta} reveals deeper information about this distribution function. Later this idea was explored further by Montgomery~\cite{Mo3} and then by Monach~\cite{Mon}. Soon after Wintner's articles, Ingham proved that Conjecture \ref{LI zeta} implies that the Mertens conjecture is false.

Unfortunately, theoretical results that provide direct evidence for LI or for Conjecture~\ref {LI zeta} are sorely lacking. Moreover, the only numerical study of such linear relations of which we are aware is Bateman {\it et al.}~\cite{Betal}, who refined Ingham's result by showing that it suffices to consider only integer linear relations $a_1\gamma_1 + \cdots + a_{n}\gamma_{n}=0$ where each $|a_j|\le 1$ except possibly for a single $|a_j|=2$. They verified numerically that no such relations exist among the smallest twenty ordinates $\gamma_j$ of zeros of $\zeta(s)$.

However, there are several interesting consequences of LI that are more amenable to analysis. For example, any set containing 0 is linearly dependent, and so LI implies:

\begin{subconj}  \label{central point}
$L(\frac12,\chi) \ne 0$ for all Dirichlet characters $\chi$.
\end{subconj}

There are several impressive partial results in the direction of Conjecture~\ref {central point}. Jutila~\cite{Ju} proved that there are infinitely many characters $\chi$ (quadratic, with prime conductor) for which $L(\frac12,\chi)\ne0$; later Balasubramanian and Murty~\cite{BM} showed that $L(\frac12,\chi)\ne0$ for at least $\frac1{25}$ of the Dirichlet characters $\chi$ with prime conductor. For general conductors $q$, Iwaniec and Sarnak~\cite{IS} proved that $L(\frac12,\chi)\ne0$ for at least $\frac13$ of the even primitive characters, while in the special case of quadratic characters, Soundararajan~\cite{So} showed that $L(\frac12,\chi)\ne0$ for at least $\frac78$ of the quadratic characters with conductor divisible by~8.

When discussing the zeros of Dirichlet $L$-functions, it is typical to speak of the ``set'' of zeros but actually mean the multiset of zeros, where each zero is listed according to its multiplicity. Since any multiset with a repeated element is linearly dependent, LI has the following implication as well:

\begin{subconj}  \label{simple}
Every zero of every Dirichlet $L$-function is simple.
\end{subconj}

There is ample numerical evidence for Conjecture~\ref{simple}: extensive calculations of zeros of $\zeta(s)$ (see~\cite{Od} and~\cite{vtw}) verify that all its known zeros are simple, and similarly, calculations by Rumely~\cite{Ru} of zeros of many Dirichlet $L$-functions confirm that all their known zeros are also simple. As for theoretical evidence, 
Levinson \cite{Le} proved that at least $\frac{1}{3}$ of the zeros of $\zeta(s)$ are simple, which was later improved by Conrey \cite{C} to $\frac25$. 
Moreover, in his seminal work on the pair correlation of zeros of $\zeta(s)$, Mongomery~\cite{Mo2} proved that if the Riemann Hypothesis is true then at least $\frac23$ of the zeros of $\zeta(s)$ are simple; furthermore, his pair correlation conjecture implies that almost all of its zeros are simple. Similar remarks apply to Dirichlet $L$-functions under the generalized pair correlation conjecture (see~\cite {otherRS}).  Unconditionally,
it seems that Bauer~\cite{Ba} was the first person to prove that a Dirichlet $L$-function has infinitely many simple zeros.
Recently, Conrey, Iwaniec, and Soundararajan have proven a technical theorem which roughly asserts that $\frac{14}{25}$ of the zeros of all Dirichlet $L$-functions are simple.  Assuming the Generalized Riemann Hypothesis,
Ozluk~\cite{Oz} proved a related theorem which also roughly asserts that $\frac{11}{12}$ of the zeros of all Dirichlet $L$-functions are simple.

(We remark in passing that under this ``multiset'' interpretation, LI actually implies the Generalized Riemann Hypothesis, since any violating zero $\beta+i\gamma$ would force another zero at $1-\beta+i\gamma$, thus causing a repetition of the imaginary part~$\gamma$. One could reformulate the LI conjecture to be independent of the truth or falsity of the Generalized Riemann Hypothesis if desired.)

We turn now to consequences of the LI conjecture to zeros in arithmetic progressions. Given a Dirichlet character $\chi$, let us define a {\em $\chi$-ordinate} to be a real number $\gamma$ such that $L(\sigma+i\gamma,\chi) = 0$ for some $0<\sigma<1$, and a {\em Dirichlet ordinate} to be a nonnegative real number $\gamma$ that is a $\chi$-ordinate for some Dirichlet character~$\chi$.

\begin{subconj}  \label{apconj}
Let $a$ and $b$ be positive real numbers.
\begin{enumerate}
\item For any Dirichlet character $\chi$, at most two points in the inhomogeneous arithmetic progression $\{ a+kb\colon k\in \Z\}$ are $\chi$-ordinates. When $a=0$, at most one point in the homogeneous arithmetic progression $\{ kb\colon k\in \Z\}$ is a $\chi$-ordinate.
\item No real number is a $\chi$-ordinate for two different Dirichlet characters $\chi$.
\end{enumerate}
\end{subconj}

Conjecture~\ref {apconj}(a) follows from the fact that any three distinct elements of an arithmetic progression are linearly dependent (or even two distinct elements, if the arithmetic progression contains 0). Conjecture~\ref {apconj}(b) is a result of interpreting LI as discussing the multiset of Dirichlet ordinates. One important special case of Conjecture~\ref {apconj}(a) comes from restricting to the central line:

\begin{subconj}  \label{relevant conj}
Let $a$ and $b$ be positive real numbers, and let $\chi$ be a Dirichlet character. At most two points in the vertical arithmetic progression $\{ \frac12+i(a+kb)\colon k\in \Z\}$ are zeros of $L(s,\chi)$.
\end{subconj}

\noindent Theorem~\ref{main theorem} goes some distance towards proving Conjecture~\ref{relevant conj} and hence serves as actual theoretical evidence in favor of the LI conjecture.

We conclude this discussion with an observation of Silberman~\cite{Silb}, who singles out the following seemingly simple consequence of LI:

\begin{subconj}  \label{ridiculous}
There does not exist a real number $\omega$ such that every Dirichlet ordinate is a rational multiple of $\omega$.
\end{subconj}

\noindent However, not only is Conjecture~\ref{ridiculous} apparently intractable at the moment, it is unknown whether any Dirichlet $L$-function at all (including the Riemann $\zeta$-function) has even a single nontrivial zero whose imaginary part is irrational!

In Section~\ref{strategy} we describe the strategy of our proofs of Theorems~\ref{main theorem} and~\ref{lowest nonzero theorem}: in short, we multiply $L\big(\frac12+i(a+kb),\chi\big)$ by a mollifier and investigate the first and second moments of the resulting mollified sums. The main tools in these proofs are the approximate functional equation for $L(s,\chi)$, a suitable version of which is proved in Section~\ref{aux results section}, and a theorem of Bauer~\cite{Ba} (see Lemma~\ref{Bauer lemma}) concerning mollified second moments of $L(s,\chi)$ integrated on the critical line. The error terms that result from the approximate functional equation depend in a slightly delicate way on the distance between $(b \log n)/2\pi$ and the nearest integer; Section~\ref{technical stuff} contains a few technical lemmas needed to handle these terms. With these preliminaries accomplished, we prove Propositions~\ref {main thm first moment prop} and~\ref {main thm second moment prop}, from which Theorem~\ref{main theorem} directly follows, in Section~\ref{main proof}. (It seems likely that we could achieve the lower bound in Theorem~\ref{main theorem} without using a mollifier at all, but the use of the mollifier simplifies the argument; for example, it allows us to get an asymptotic formula for the first moment that is uniform in the spacing $b$ between points in the arithmetic progression.) The technically simpler Theorem~\ref{lowest nonzero theorem} is proved in Section~\ref{lowest nonzero proof}.

(Note added in proof: after this manuscript was submitted, we learned that Li and Radziwi\l\l~\cite{LR} have shown that at least $(\frac13+o(1))T$ points in a vertical arithmetic progression of length $T$ are not zeros of the Riemann $\zeta$-function, and they mention that their method extends to Dirichlet $L$-functions as well.)

\section{Moments and mollifiers}
\label{strategy}

In this section we describe our strategy for establishing Theorems~\ref{main theorem} and~\ref{lowest nonzero theorem}. As is not uncommon when investigating the nonvanishing of functions defined by Dirichlet series, both proofs depend upon sums of values of $L(s,\chi)$ weighted by a mollifier of the form
\begin{equation}
  \label{eq:mollifier}
   M(s,\chi,P) = \sum_{1 \le n \le X} \frac{\mu(n) \chi(n)}{n^s} P \Big(1-\frac{\log n}{\log X} \Big).
\end{equation}
In this definition, $X$ is a positive parameter, and $P$ can be any polynomial; the proofs of Theorems~\ref{main theorem} and~\ref{lowest nonzero theorem} employ the simple polynomials $P_1(x) = x$ and $P_2(x)=1$, respectively.

We introduce the notation
\[
   s_k = \tfrac{1}{2} + 2 \pi i (\alpha + k \beta)
\]
for integers $k$. (This notation differs from that of Theorems~\ref{main theorem} and~\ref{lowest nonzero theorem} by the change of variables $a= 2 \pi \alpha$ and $b= 2 \pi \beta$; this normalization simplifies a number of formulas that will be encountered later.) We define the corresponding first and second moments
\[
   S_1(T,P) = \sum_{k=1}^T L(s_k, \chi) M(s_k,\chi,P) 
\]
and 
\[
   S_2(T,P) = \sum_{k=1}^T |L(s_k, \chi) M(s_k,\chi,P)|^2,
\]
both functions of a positive parameter~$T$. A common application of the Cauchy--Schwarz inequality reveals that
\[
  \# \{ 1 \le k \le T \colon L(s_k,\chi) \ne 0 \} \ge 
   \frac{|S_1(T,P_1)|^2}{S_2(T,P_1)};
\]
therefore to establish Theorem~\ref{main theorem}, we seek a lower bound for the first moment $S_1(T,P_1)$ and an upper bound for the second moment $S_2(T,P_1)$. In fact we can establish that
\begin{equation}
   \label{eq:1st moment bound}
  S_1(T,P_1) \sim T
\end{equation}
(see Proposition~\ref {main thm first moment prop}) and
\begin{equation}
    \label{eq:2nd moment bound}
   S_2(T,P_1)  \ll T \log T
\end{equation}
(see Proposition~\ref {main thm second moment prop}), which combine to yield 
\[
    \# \{ 1 \le k \le T \colon L(s_k,\chi) \ne 0 \} \gg \frac{T}{\log T}
\]
and hence Theorem~\ref{main theorem}. In the proof, the length $X$ of the Dirichlet polynomial defining $M(s,\chi,P)$ is taken to be a power of $T$, namely $X = T^{\theta}$: the asymptotic formula~\eqref {eq:1st moment bound} holds for any $0< \theta <1$, while the estimate~\eqref {eq:2nd moment bound} holds for any $0< \theta <\frac12$. In summary, this discussion shows that we have reduced the proof of Theorem~\ref{main theorem} to verifying Propositions~\ref {main thm first moment prop} and~\ref {main thm second moment prop}.

The reader who is familiar with the use of mollifiers might be surprised to see that the method just described does not yield a positive proportion of nonzero values of $L(s,\chi)$. Indeed, we tried several alternate methods in the hopes of obtaining the presumably correct bound $S_2(T,P_1) \ll T$, but ultimately without success. The obstacle seems to be related to the fact that our second moment $S_2(T,P_1)$ is a discrete sum, while mollifiers tend to be most effective when used in continuous moments. In particular, we use Gallagher's lemma~\cite[Lemma 1.2]{Mo} to bound $S_2(T,P_1)$ above by a related integral, which we then evaluate asymptotically; therefore the loss of the factor of $\log T$ must occur in the use of Gallagher's lemma itself. While we were not able to realize such a plan, it is possible that an evaluation that directly addresses the discrete sum itself could show that $S_2(T,P_1) \ll T$; this stronger upper bound would improve Theorem~\ref{main theorem} to the statement that a positive proportion of the points in a vertical arithmetic progression are not zeros of $L(s,\chi)$.

This is not to say that the use of a mollifier has no benefits. For one thing, the asymptotic formula $S_1(T,P_1) \sim T$ is no longer necessarily true for the unmollified sum $\sum_{k=1}^T L(s_k,\chi)$: for example, if $\alpha=0$ and $\beta =\frac{2\pi}{\log 2}$, then one can show that
\[
\sum_{k=1}^T L( s_k,\chi ) = \sum_{k=1}^T L\bigg( \frac12 + \frac{2\pi k}{\log 2}i,\chi \bigg) \sim \bigg( 1-\frac{\chi(2)}{\sqrt 2} \bigg)^{-1} T.
\]
Even if $\beta$ were simply close to $\frac{2\pi}{\log 2}$ rather than exactly equal to it, the asymptotic formula $\sum_{k=1}^T L(s_k,\chi) \sim T$ would not be true uniformly in~$\beta$. In summary, the presence of the mollifier allows us to establish the asymptotic formula~\eqref{eq:1st moment bound} uniformly and hence with less technical annoyance. Furthermore, the mollified integral that arises after using Gallagher's lemma to bound $S_2(T,P_1)$ is one that has already been evaluated by Bauer~\cite{Ba}, which saves us trouble in the second moment as well.


The proof of Theorem~\ref{lowest nonzero theorem} is conceptually much simpler, since all we need to show is that
\[
   S_1(T,P_2) =  \sum_{k=1}^T L(s_k, \chi) M(s_k,\chi,P_2) \ne 0
\]
for $T$ sufficiently large as a function of $\beta$ and the conductor $q$ of~$\chi$. In fact, we show that $S_1(T,P_2) = T+o(T)$, where the error term is bounded by an explicit function of $T$, $\beta$, and~$q$. It then suffices to determine the least $T$ such that the error term is smaller than $T$. 
  
We briefly mention how our work relates to previous arguments. 
Lapidus and van Frankenhuysen~\cite[Chapter 9]{LvF} consider the expression
\[
 \sum_{k=-\infty}^{\infty} \frac{x^{s_k }}{s_k } L(s_k, \chi)
\] 
where  $s_k=\tfrac{1}{2}+2 \pi i k \beta$ and $x$ and is a real variable.
They developed an explicit formula for this expression which allowed them to deduce non-vanishing results for
$L(s_k, \chi)$.  Watkins's bound \eqref{eq:bound for k} is derived from an explicit formula for a variant of the above sum. 
We observe that all of the arguments rely on the calculation of certain discrete means of $L(s_k,\chi)$. 


\section{Truncated Dirichlet series}
\label{aux results section}


It is crucial to our method that $L(s,\chi)$ can be well-approximated inside the critical strip by a suitable truncation of its defining Dirichlet series. 
Although such a result is well-known for the Riemann zeta function, we were unable to find the corresponding result for Dirichlet $L$-functions in the literature. 
In fact, Corollary 1 of~\cite{Ra} states such a result; unfortunately,  no proof is provided and it does not follow from the main theorem in that article.
The following proposition provides a truncated Dirichlet series representation for $L(s,\chi)$ and is a $q$-analogue 
of \cite[Theorem 4.11]{T}.   Moreover, it improves the dependence on $q$ (particularly for imprimitive characters) and makes explicit the dependence on~$s$
of~\cite[Corollary 1]{Ra}.   In order to keep this article self-contained and to make available a reference for other researchers, we provide a proof. 
Throughout we use the traditional notation $s=\sigma+it$ to name the real and imaginary parts of~$s$.

\begin{prop}   
\label{afedirichlet}
Let $\chi$ be a non-principal Dirichlet character modulo $q$, and let $C > 1$ be a constant. Uniformly for all complex numbers $s$ with $\sigma>0$ and all real numbers $x >  C |t|/2 \pi$,
\begin{equation}
  \label{eq:truncatedDS}
L(s,\chi) = \sum_{n \le qx} \chi(n)n^{-s} + O_C\bigg( q^{1/2-\sigma} x^{-\sigma} (\log q) \bigg( 1 + \min\bigg\{ \frac\sigma{x\log q}, \frac{|t|}\sigma \bigg\} \bigg) \bigg).
\end{equation}
\end{prop}

The dependence upon $C$ of the constant implicit in the $O_C$ notation is quite mild, something like $\min\{1,1/(C-1)\}$; we emphasize that the implicit constant does not depend upon the character~$\chi$. 

We begin with two lemmas that will be used in service of the important formula given in Lemma~\ref {L from Hurwitz lemma} below. We use the standard notation $\zeta(s,\alpha) = \sum_{n=0}^{\infty} (n+\alpha)^{-s}$ for the Hurwitz zeta function (valid for $\sigma>1$).

\begin{lemma}
\label{Hurwitz lemma}
Let $\alpha$, $x$, and $N$ be real numbers with $x>\alpha>0$ and $N>x$. Let $s$ be a complex number with $\sigma>0$. Then
\begin{equation*}
   \zeta(s,\alpha) = \sum_{0 \le n \le x -\alpha} (n+\alpha)^{-s} + \frac{ \{ x-\alpha \} - 1/2}{x^s} - \frac{x^{1-s}}{1-s} -s \int_{x}^{N} \frac{ \{ u-\alpha \} - 1/2}{u^{s+1}} \,du + O \bigg( \frac{|s|}\sigma N^{-\sigma}\bigg). 
\end{equation*}
\end{lemma}

\begin{proof}
We use the following simple summation-by-parts formula \cite[p.~13]{T}: if 
$\phi$ is a function with a continuous derivative on $[a,b]$, then
\begin{equation}
 \label{eq:summationformula}
  \sum_{a < n \le b} \phi(n) = \int_{a}^{b} \phi(u) \,du + \int_{a}^{b} (\{ u \} - \tfrac{1}{2}) \phi'(u) \,du + (\{ a \} - \tfrac{1}{2})\phi(a)- (\{b \} - \tfrac{1}{2}) \phi(b).
\end{equation}
Applying this formula with $\phi(n) = (n+\alpha)^{-s}$ on the interval $(x-\alpha,M]$ for a large integer $M$, we have
\begin{align*}
\sum_{0\le n\le M} (n+\alpha)^{-s} &= \sum_{0\le n\le x-\alpha} (n+\alpha)^{-s} + \int_{x-\alpha}^M (u+\alpha)^{-s} \,du \\
&\qquad{}- s \int_{x-\alpha}^M \frac{\{ u \} - 1/2}{(u+\alpha)^{s+1}} \,du + \frac{\{ x-\alpha \} - 1/2}{x^s} + \frac1{2(M+\alpha)^s} \\
&= \sum_{0\le n\le x-\alpha} (n+\alpha)^{-s} + \frac{(M+\alpha)^{1-s}}{1-s} - \frac{x^{1-s}}{1-s} \\
&\qquad{}- s \int_x^{M+\alpha} \frac{\{ u-\alpha \} - 1/2}{u^{s+1}} \,du + \frac{\{ x-\alpha \} - 1/2}{x^s} + \frac1{2(M+\alpha)^s}.
\end{align*}
Taking the limit of both sides as $M$ goes to infinity, which is valid for $\sigma>1$, we see that
\begin{equation}
\label{analytically continue}
\zeta(s,\alpha) = \sum_{0\le n\le x-\alpha} (n+\alpha)^{-s} - \frac{x^{1-s}}{1-s} + \frac{\{ x-\alpha \} - 1/2}{x^s} - s \int_x^\infty \frac{\{ u-\alpha \} - 1/2}{u^{s+1}} \,du.
\end{equation}
This establishes the lemma in the region $\sigma >1$. However, we observe that the last integral is bounded by  $O(N^{-\sigma}/\sigma)$; consequently, the right-hand side is a meromorphic function of $s$ for $\sigma>0$. As the left-hand side is 
also a meromorphic function of $s$, it follows that equation~\eqref{analytically continue} is valid in the half-plane $\sigma >0$ by analytic continuation.
\end{proof}
At this point we introduce two other standard pieces of notation: the complex exponential $e(y) = e^{2\pi i y}$ of period 1, and the Gauss sum $\tau(\chi) = \sum_{a=1}^{q} \chi(a) e(a/q)$.
The following lemma provides the main formula upon which rests our proof of Proposition~\ref {afedirichlet}.
\begin{lemma}
\label{L from Hurwitz lemma}
Let $\alpha$, $x$, and $N$ be real numbers satisfying $0<x<N$, let $s$ be a complex number with $\sigma>0$, and let $q>1$ be an integer. For any primitive character $\chi\mod q$,
\begin{multline}
  \label{Lschiexpression}
L(s,\chi) = \sum_{n \le qx} \chi(n)n^{-s} + (qx)^{-s} \sum_{a=1}^{q} \chi(a) \big( \{ x-\tfrac aq \} - \tfrac12 \big) \\
+ \frac{sq^{-s}}{2\pi i} \tau(\chi) \sum_{\nu\in\Z\setminus\{0\}} \frac{\bar\chi(-\nu)}\nu \int_x^N \frac{e(\nu u)}{u^{s+1}} \,du + O \bigg( \frac{|s|}\sigma N^{-\sigma} q^{1-\sigma} \bigg).
\end{multline}
\end{lemma}

\begin{proof}
The identity
$L(s,\chi) = q^{-s} \sum_{a=1}^{q} \chi(a) \zeta \big(s,\tfrac{a}{q} \big)$ is well known~\cite[page 71]{D}.
Using Lemma~\ref{Hurwitz lemma}, we proceed by writing
\begin{align}
L(s,\chi) 
&= q^{-s} \sum_{a=1}^{q} \chi(a) \sum_{0 \le n \le x -a/q} \big(n+\tfrac aq \big)^{-s} + (qx)^{-s} \sum_{a=1}^{q} \chi(a) \big( \{ x-\tfrac aq \} - \tfrac12 \big) \notag \\
&\qquad{}-q^{-s} \sum_{a=1}^{q} \chi(a) \frac{x^{1-s}}{1-s} - s q^{-s} \sum_{a=1}^{q} \chi(a) \int_{x}^{N} \frac{ \{ u-a/q \} - 1/2 }{u^{s+1}} \,du + O \bigg( \frac{|s|}\sigma N^{-\sigma} q^{1-\sigma} \bigg),
\label{middle term vanished}
\end{align}
where the term $-q^{-s} \sum_{a=1}^{q} \chi(a) \frac{x^{1-s}}{1-s}$ vanishes because $\chi$ is non-principal (here we use the assumption $q>1$). The first term on the right-hand side of equation~\eqref{middle term vanished} can be simplified using the change of variables $m=qn+a$:
\begin{equation}
  \label{eq:firstsimplification}
q^{-s} \sum_{a=1}^{q} \chi(a) \sum_{0 \le n \le x- {a/q}} \big(n+ \tfrac{a}{q} \big)^{\!{-}s} = \sum_{1\le m \le qx} \chi(m)m^{-s}.
\end{equation}

We now simplify the third term on the right-hand side of equation~\eqref{middle term vanished}.
The standard Fourier series
$\{ y \} - \frac{1}{2} = -\frac{1}{2 \pi i} \sum_{\nu\in\Z\setminus\{0\}} \frac{e(\nu y)}{\nu}$
is valid for all non-integer real numbers~$y$. 
By the bounded convergence theorem we may exchange the integral and sum, obtaining
\begin{equation*}
- \int_{x}^{N} \frac{ \{ u-\alpha \} - {1/2} }{u^{s+1}} \,du = \frac{1}{2 \pi i} \sum_{\nu\in\Z\setminus\{0\}} \frac{e(-\nu \alpha)}{\nu} \int_{x}^{N} \frac{e(\nu u)}{u^{s+1}} \,du.
\end{equation*}
We now set $\alpha=\frac aq$, multiply both sides by $sq^{-s}\chi(a)$, and sum over $a$ to obtain
\begin{equation}
\begin{split}
   \label{eq:secondsimplification}
{-}sq^{-s} \sum_{a=1}^q \chi(a) \int_{x}^{N} \frac{ \{ u-a/q \} - {1/2} }{u^{s+1}} \,du & 
 = \frac{sq^{-s}}{2 \pi i} \sum_{\nu\in\Z\setminus\{0\}} \bigg( \frac1{\nu} \int_{x}^{N} \frac{e(\nu u)}{u^{s+1}} \,du \bigg) \bar\chi(-\nu) \tau(\chi)
\end{split}
\end{equation}
since $ \sum_{a=1}^q \chi(a) e(-\nu a/q)=  \bar\chi(-\nu) \tau(\chi)$ for
 $\chi$ primitive and
any integer $\nu$~\cite[Corollary 9.8]{MV}.
Inserting the evaluations~\eqref{eq:firstsimplification} and~\eqref{eq:secondsimplification} into equation~\eqref{middle term vanished} establishes the lemma.
\end{proof}

We must establish two somewhat technical estimates before arriving at the proof of Proposition~\ref {afedirichlet}. For these estimates, we introduce the notation $S_\chi(u) = \sum_{a\le u} \chi(a)$, and we recall the P\'olya--Vinogradov inequality $S_\chi(u) \ll q^{1/2}\log q$ (see~\cite[Theorem 9.18]{MV}).

\begin{lemma}
\label{bound sum lemma}
For any non-principal character $\chi\mod q$, any positive real number $x$, and any complex number $s$,
\begin{equation*}
  (qx)^{-s} \sum_{a=1}^{q} \chi(a) 
   ( \{ x- \tfrac{a}{q} \} - \tfrac{1}{2}  )  \ll  q^{1/2-\sigma} x^{-\sigma} \log q.
\end{equation*}
\end{lemma}

\begin{proof}
It suffices to show that the sum itself is  $O(\sqrt q \log q)$ uniformly in $x$. 
Since the sum is a periodic function of $x$ with period 1, we may assume that $0\le x<1$. Setting $b=\lfloor qx \rfloor$, we have
\begin{equation}
\begin{split}
 \label{eq:mainsum}
\sum_{a=1}^{q} \chi(a) \big( \{ x- \tfrac{a}{q} \}  - \tfrac12 \big) &= \sum_{a=1}^{b} \chi(a) \big( x- \tfrac{a}{q}  - \tfrac12 \big) + \sum_{a=b+1}^{q} \chi(a) \big( x- \tfrac{a}{q} +1  - \tfrac12 \big) \\
&= \sum_{a=b+1}^q \chi(a) - \tfrac1q \sum_{a=1}^q a\chi(a) + \big( x-\tfrac12 \big) \sum_{a=1}^q \chi(a) \\
&= - \sum_{a=1}^{b} \chi(a) - \tfrac1q \sum_{a=1}^q a\chi(a),
\end{split}
\end{equation}
since $\sum_{a=1}^q \chi(a) = 0$. By partial summation we obtain
\[
   \sum_{a=1}^q a\chi(a) = q S_{\chi}(q) -\int_{1}^{q} S_{\chi}(t)dt \ll q^{3/2} \log q
\]
by the P\'olya--Vinogradov inequality. We deduce that the sum in equation~\eqref{eq:mainsum}
is $O(\sqrt{q} \log q)$ as desired. 
\end{proof}

\begin{lemma}
\label{bound integrals lemma}
Let $\chi\mod q$ be a primitive character, and let $s$ be a complex number. Let $C>1$ be a constant, and let $N$ and $x$ be real numbers with $N>x>C|t|/2\pi$. Then
\begin{equation*}
\frac{sq^{-s}}{2\pi i} \tau(\chi) \sum_{\nu\in\Z\setminus\{0\}} \frac{\bar\chi(-\nu)}\nu \int_x^N \frac{e(\nu u)}{u^{s+1}} \,du\ll_C |s| q^{1/2-\sigma} x^{-\sigma-1}.
\end{equation*}
\end{lemma}

\begin{proof}
The main step of the proof is to show that
\begin{equation}
\label{just the integral bound}
\int_x^N \frac{e(\nu u)}{u^{s+1}} \,du \ll \frac1{x^{\sigma+1}} \frac1{|\nu|-C^{-1}}
\end{equation}
for any nonzero integer~$\nu$.
Indeed, if we have the estimate~\eqref{just the integral bound} then
\begin{align*}
\frac{sq^{-s}}{2\pi i} \tau(\chi) \sum_{\nu\in\Z\setminus\{0\}} \frac{\bar\chi(-\nu)}\nu \int_x^N \frac{e(\nu u)}{u^{s+1}} \,du &\ll |s| q^{1/2-\sigma} \sum_ {\nu\in\Z\setminus\{0\}} \frac1{|\nu|} \cdot \frac1{x^{\sigma+1}} \frac1 {|\nu|-C^{-1}} \\
&= |s| q^{1/2-\sigma} x^{-\sigma-1} \sum_{\nu=1}^\infty \frac1{\nu(\nu-C^{-1})} \ll_C |s| q^{1/2-\sigma} x^{-\sigma-1}
\end{align*}
by the Gauss sum evaluation $|\tau(\chi)| = q^{1/2}$~\cite[Theorem 9.7]{MV}.

To establish the estimate~\eqref{just the integral bound}, we may assume without loss of generality that $t\ge0$, since conjugating the left-hand side is equivalent to replacing $t$ by $-t$ and $\nu$ by $-\nu$. We begin by writing
\begin{align*}
\int_x^N \frac{e(\nu u)}{u^{s+1}} \,du = \int_x^N \frac{e^{2\pi i\nu u - it\log u}}{u^{\sigma+1}} \,du = -i \int_x^N \big( 2\pi i\nu-\tfrac{it}u \big) e^{2\pi i\nu u - it\log u} G(u) \,du,
\end{align*}
where we have set $G(u) = 1/\big( u^{\sigma}(2\pi\nu u-t) \big)$. Integrating by parts,
\begin{align*}
\int_x^N \frac{e(\nu u)}{u^{s+1}} \,du &= -i e^{i(2\pi \nu u - t\log u)} G(u) \bigg|_x^N + i \int_x^N e^ {i(2\pi \nu u - t\log u)} G'(u) \,du \\
&\ll |G(N)| + |G(x)| + \int_x^N | G'(u) | \,du.
\end{align*}
It is easy to check that $G(u)$ is positive and decreasing for $u>t/2\pi$ when $\nu>0$, while $G(u)$ is negative and increasing for $u>0$ when $\nu<0$. Since $x>Ct/2\pi$ and $C>1$, we conclude that $|G(u)|$ is decreasing for $u\ge x$. Therefore 
\begin{align*}
\int_x^N \frac{e(\nu u)}{u^{s+1}} \,du &\ll |G(x)| + \int_x^\infty \big( {-}\mathop{\rm sgn}(\nu) G'(u) \big) \,du 
= 2|G(x)|  \\
&= \frac2{x^\sigma} \frac1{|2\pi\nu x-t|} 
< \frac2 {x^\sigma} \frac1{2\pi|\nu| x- 2\pi C^{-1}x} \ll \frac1 {x^{\sigma+1}} \frac1{|\nu|-C^{-1}}
\end{align*}
as claimed.
\end{proof}

\begin{proof}[Proof of Proposition~\ref{afedirichlet}]
We begin by assuming that $\chi$ is primitive and $q>1$, so that the statement of Lemma~\ref {L from Hurwitz lemma} is valid:
\begin{multline*}
L(s,\chi) = \sum_{n \le qx} \chi(n)n^{-s} + (qx)^{-s} \sum_{a=1}^{q} \chi(a) \big( \{ x-\tfrac aq \} - \tfrac12 \big) \\
+ \frac{sq^{-s}}{2\pi i} \tau(\chi) \sum_{\nu\in\Z\setminus\{0\}} \frac{\bar\chi(-\nu)}\nu \int_x^N \frac{e(\nu u)}{u^{s+1}} \,du + O \bigg( \frac{|s|}\sigma N^{-\sigma} q^{1-\sigma} \bigg).
\end{multline*}
Using Lemma~\ref{bound sum lemma} and Lemma~\ref{bound integrals lemma} to bound the second and third terms on the right-hand side yields
\[
L(s,\chi) = \sum_{n \le qx} \chi(n)n^{-s} + O\big( q^{1/2-\sigma} x^{-\sigma} \log q \big) + O_C\big( |s| q^{1/2-\sigma} x^{-\sigma-1} \big) + O \bigg( \frac{|s|}\sigma N^{-\sigma} q^{1-\sigma} \bigg).
\]
By letting $N \to \infty$ and observing that $|s|/x \ll (\sigma+|t|)/x \ll \sigma/x + 1$, we obtain
\begin{equation}
\label{eq:first form}
  L(s,\chi)  = \sum_{n \le qx} \chi(n) n^{-s} + O_C\bigg( q^{1/2-\sigma} x^{-\sigma} (\log q) \bigg( 1 + \frac\sigma{x\log q} \bigg) \bigg).
\end{equation}

In order to prove the stronger error term in equation~\eqref{eq:truncatedDS}, we shall establish a second bound for $L(s,\chi)- \sum_{n\le qx} \chi(n)n^{-s}=  \sum_{n>qx} \chi(n)n^{-s}$. 
For $\sigma > 1$, it follows from partial summation that
\begin{equation}
  \label{eq:Lschisum}
  \sum_{n>qx} \chi(n)n^{-s}
= -\frac{S_{\chi}(qx)}{(qx)^{s}}+s \int_{qx}^{\infty} \frac{S_{\chi}(u)}{u^{s+1}} \,du;
\end{equation}
by the  P\'olya--Vinogradov inequality,
\[
     \sum_{n>qx} \chi(n)n^{-s} 
  \ll \frac{q^{1/2}\log q}{(qx)^\sigma} + |s| \int_{xq}^\infty \frac{q^{1/2}\log q}{u^{\sigma+1}} \,du
  \ll q^{1/2-\sigma}x^{-\sigma}(\log q) \bigg( 1 + \frac{|s|}\sigma \bigg).
\]
Hence, the right-hand side of equation~\eqref{eq:Lschisum} may be analytically continued to $\sigma >0$. 
Since $|s|/\sigma \ll (\sigma+|t|)/\sigma = 1+|t|/\sigma$, we obtain 
\begin{equation}
\begin{split}
 \label{eq:second form}
 L(s,\chi) 
 & = \sum_{n\le qx} \chi(n)n^{-s} + O\bigg( q^{1/2-\sigma}x^{-\sigma}(\log q) \bigg( 1 + \frac{|t|}\sigma \bigg) \bigg). 
\end{split}
\end{equation}
The combination of the estimates~\eqref{eq:first form} and~\eqref{eq:second form} establishes the proposition when $\chi$ is primitive.


Now let $\chi\mod q$ be any non-principal character; we know that $\chi$ is induced by some character $\chi^*\mod{q^*}$, where $q^*\mid q$ and $q^*>1$. Define $r$ to be the product of the primes dividing $q$ but not~$q^*$. Then $\chi(n)$ equals $\chi^*(n)$ if $(n,r)=1$ and 0 if $(n,r)>1$. Consequently,
\[
\sum_{n\le qx} \chi(n)n^{-s} = \sum_{\substack{n\le qx \\ (n,r)=1}} \chi^*(n)n^{-s} = \sum_{n\le qx} \chi^*(n)n^{-s} \sum_{d\mid (n,r)} \mu(d)
\]
by the characteristic property of the M\"obius $\mu$-function. Reversing the order of summation,
\begin{equation}
\label{will apply twice}
\sum_{n\le qx} \chi(n)n^{-s} = \sum_{d\mid r} \mu(d) \sum_{\substack{n\le qx \\ d\mid n}} \chi^*(n)n^{-s} = \sum_{d\mid r} \frac{\mu(d)\chi^*(d)}{d^s} \sum_{m\le qx/d} \chi^*(m)m^{-s}.
\end{equation}
Since $\chi^*\mod{q^*}$ is primitive, we may use the already established equation~\eqref{eq:truncatedDS} on the right-hand side of equation~\eqref{will apply twice}, with $qx/q^*d$ in place of $x$; we obtain
\begin{align*}
\sum_{n\le qx} & \chi(n)n^{-s} \\
&= \sum_{d\mid r} \frac{\mu(d)\chi^*(d)}{d^s} \bigg( L(s,\chi^*) + O_C\bigg( (q^*)^{1/2-\sigma} \Big( \frac{qx}{q^*d} \Big)^{-\sigma} (\log q) \bigg( 1 + \min\bigg\{ \frac{q^*d\sigma}{qx\log q}, \frac{|t|}{\sigma} \bigg\} \bigg) \bigg) \bigg) \\
&= L(s,\chi^*) \sum_{d\mid r} \frac{\mu(d)\chi^*(d)}{d^s} + O_C\bigg( (q^*)^{1/2}q^{-\sigma} x^{-\sigma} (\log q) \bigg( 1 + \min\bigg\{ \frac{\sigma}{x\log q}, \frac{|t|}{\sigma} \bigg\} \bigg) \sum_{d\mid r} 1 \bigg) \\
&= L(s,\chi^*) \prod_{p\mid r} \bigg( 1 - \frac{\chi^*(p)}{p^s} \bigg) + O_C\bigg( 2^{\omega(r)} (q^*)^{1/2}q^{-\sigma} x^{-\sigma} (\log q) \bigg( 1 + \min\bigg\{ \frac{\sigma}{x\log q}, \frac{|t|}{\sigma} \bigg\} \bigg) \bigg).
\end{align*}
Note that $L(s,\chi^*) \prod_{p\mid r} \big( 1 - {\chi^*(p)/p^s} \big) = L(s,\chi)$; thus to establish equation~\eqref{eq:truncatedDS} for all non-principal characters, it suffices to show that $2^{\omega(r)} (q^*)^{1/2} \ll q^{1/2}$ in the error term. But this follows from the calculation
\[
2^{\omega(r)} \frac{ (q^*)^{1/2} }{q^{1/2} } \le \prod_{p\mid r} 2 \cdot \prod_{\substack{p^\alpha\parallel q \\ p\mid r}} (p^\alpha)^{-1/2} \le \prod_{p\mid r} (2p^{-1/2}) \le \frac2{\sqrt2}\frac2{\sqrt3},
\]
and so the proof of the proposition is complete.
\end{proof}

\section{Sums involving the distance to the nearest integer}
\label{technical stuff}

In our proofs of Theorems~\ref{main theorem} and~\ref {lowest nonzero theorem}, we need to estimate sums of the shape $\sum \min(T, \| \beta\log n \|^{-1})$ over certain sets of natural numbers. An important point is that the summand is very large when $n$ is close to one of the integers
\[
  E_j = \big\lfloor e^{j/\beta} \big\rfloor.
\]
For this reason, we begin by bounding the above sum when $n$ ranges from $E_j$ to $E_{j+1}$. We have made the effort to keep explicit all dependences of the implicit constants on $\beta$, particularly with an eye towards the application to Theorem~\ref {lowest nonzero theorem}.
 
\begin{lemma}
\label{one beta interval lemma}
Let $T \ge 1$ and $\beta > 0$ be real numbers. Uniformly for all integers $j\ge0$,
\[
\sum_{E_j\le n\le E_{j+1}} \min\{ T, \| \beta\log n \|^{-1} \} \ll T +  \frac{e^{1/\beta}}{\beta} \frac{j+1}{\beta} e^{j/\beta}.
\]
\end{lemma}

\begin{proof}
For convenience we will split the interval $[E_j,E_{j+1}]$ into two intervals at the point $H_j = \big\lfloor e^{(j+1/2)/\beta} \big\rfloor$. For the first half of the interval, note that
\begin{align*}
  \sum_{E_j\le n\le H_j} \min\{ T, \| \beta\log n \|^{-1} \} &\le 2T + \sum_{E_j+2\le n\le H_j} (\beta\log n - j)^{-1} \notag \\
  & \ll T + \tfrac1\beta \sum_{2\le h\le H_j-E_j} \big(\log (E_j + h) - \tfrac j\beta\big)^{-1},
\end{align*}
where we have set $h=n-E_j$. Since
\[
  \log(E_j + h) - \frac j\beta \ge \log\big( e^{j/\beta} - 1 + h \big) - \frac j\beta = \log \bigg( 1 + \frac{h-1}{e^{j/\beta}} \bigg),
\]
the last upper bound becomes
\begin{equation}
  \sum_{E_j\le n\le H_j} \min\{ T, \| \beta\log n \|^{-1} \} \ll T + \tfrac1\beta \sum_{2\le h\le H_j-E_j} \bigg( \log \bigg( 1 + \frac{h-1}{e^{j/\beta}} \bigg) \bigg)^{-1}.
\label{two terms and rest}
\end{equation}
The term $(h-1)/e^{j/\beta}$ can be bounded above by
\[
\frac{h-1}{e^{j/\beta}} \le \frac{H_j-E_j-1}{e^{j/\beta}} \le \frac{e^{(j+1/2)/\beta} - (e^{j/\beta}-1) -1}{e^{j/\beta}} = e^{1/2\beta} - 1;
\]
by the concavity of $\log(1+x)$ on the interval $0\le x\le e^{1/2\beta} - 1$, we infer that
\[
\log \bigg( 1 + \frac{h-1}{e^{j/\beta}} \bigg) \ge \frac{h-1}{e^{j/\beta}} \frac{\log(1+(e^{1/2\beta} - 1))}{e^{1/2\beta} - 1} = \frac{h-1}{2\beta e^{j/\beta} (e^{1/2\beta} - 1)}.
\]
Therefore equation~\eqref{two terms and rest} becomes
\begin{align}
  \sum_{E_j\le n\le H_j} \min\{ T, \| \beta\log n \|^{-1} \} &\ll T + \tfrac1\beta \sum_{2\le h\le H_j-E_j} \frac{\beta e^{j/\beta} (e^{1/2\beta} - 1)}{h} \notag \\
  &\le T + e^{j/\beta} (e^{1/2\beta} - 1) \log(H_j-E_j) \notag \\
  &\le T + e^{j/\beta} (e^{1/2\beta} - 1) \log e^{(j+1/2)/\beta} = T + (e^{1/2\beta} - 1) \frac{j+1/2}\beta e^{j/\beta}.
  \label{first half bound}
\end{align}

Similarly, consider
\begin{equation*}
  \sum_ {H_j\le n\le E_{j+1}} \min\{ T, \| \beta\log n \|^{-1} \} \le 2T + \sum_{H_j\le n\le E_{j+1}-2} (j+1-\beta\log n)^{-1}.
\end{equation*}
Note that for $2\le h\le E_{j+1}-H_j$, by the convexity of $\log(1-x)^{-1}$,
\begin{equation*}
  \frac{j+1}\beta - \log(E_{j+1} - h) \ge \frac{j+1}\beta - \log(e^{(j+1)/\beta} -h \big) = \log \bigg( 1 - \frac h{e^{(j+1)/\beta}} \bigg)^{-1} \ge \frac h{e^{(j+1)/\beta}}.
\end{equation*}
Therefore
\begin{align}
 \sum_ {H_j\le n\le E_{j+1}} \min\{ T, \| \beta\log n \|^{-1} \} &\ll T + \frac{e^{1/\beta}}{\beta} \sum_{2\le h\le E_{j+1}-H_j} \frac {e^ {j/\beta}} h \notag \\
 &\le T + \frac{e^{1/\beta}}{\beta} e^{j/\beta}  \log (E_{j+1} - H_j) \notag \\
 &\le T + \frac{e^{1/\beta}}{\beta} e^{j/\beta} \log e^{(j+1)/\beta} = T + \frac{e^{1/\beta}}{\beta} \frac{j+1}{\beta} e^{j/\beta}.
 \label{second half bound}
\end{align}
The lemma follows from adding the upper bounds~\eqref{first half bound} and~\eqref{second half bound} and noting that $e^{1/2\beta}-1 \le \tfrac1\beta e^{1/\beta}$ for $\beta >0$. 
\end{proof}

With the previous lemma in hand, we may now obtain a bound for
weighted averages of $ \min\{ T, \| \beta\log n \|^{-1} \}$
over arbitrary intervals.

\begin{prop}
For any real numbers $B > A\ge 1$,
\label{one fell swoop prop}
\begin{equation}
\label{beta gt 0}
\sum_{A< n\le B} n^{-1/2} \min\{ T, \| \beta\log n \|^{-1} \} \ll
 \frac{e^{1/\beta}}{e^{1/2\beta}-1}
  \bigg( T A^{-1/2} + \frac{e^{1/\beta}}{\beta} B^{1/2} (\log B +\beta^{-1}) \bigg)
\end{equation}
uniformly for $T \ge 1$ and $\beta > 0$. In particular, for $\beta \ge 1$,
\begin{equation}
\label{beta ge 1}
 \sum_{A\le n\le B} n^{-1/2} \min\{ T, \| \beta\log n \|^{-1} \} 
 \ll \beta T A^{-1/2} + B^{1/2} \log B.
\end{equation}
\end{prop}

\begin{proof}
If we choose nonnegative integers $I$ and $J$ such that $e^{-1/\beta}A < e^{I/\beta} \le A$ and $B \le e^{J/\beta} < Be^{1/\beta}$, we have
\begin{align*}
\sum_{A< n\le B} n^{-1/2} \min\{ T, \| \beta\log n \|^{-1} \} &\le \sum_{E_I< n\le E_J} n^{-1/2} \min\{ T, \| \beta\log n \|^{-1} \} \\
& \ll \sum_{I\le j\le J} (E_j+1)^{-1/2} \sum_{E_j < n \le E_{j+1}} \min\{ T, \| \beta\log n \|^{-1} \} \\
&\ll \sum_{I\le j\le J} e^{-j/2\beta} \bigg( T + \frac{e^{1/\beta}}{\beta} \frac{j+1}{\beta} e^{j/\beta} \bigg)
\end{align*}
by Lemma~\ref{one beta interval lemma}. This last expression is essentially the sum of two geometric series:
\begin{align*}
\sum_{A\le n\le B} n^{-1/2} \min\{ T, \| \beta\log n \|^{-1} \} &\ll T \sum_{j\ge I} e^{-j/2\beta} + \frac{e^{1/\beta}}{\beta}
\frac{J}{\beta}   \sum_{j\le J} e^{j/2\beta} \\
&\ll T \frac{e^{-I/2\beta}}{1-e^{-1/2 \beta}} + \frac{e^{1/\beta}}{\beta} \frac J\beta \frac{e^{J/2\beta}}{1-e^{-1/2 \beta}}.
\end{align*}
Since $e^{I/2\beta} > A^{1/2}e^{-1/2\beta}$ and $e^{J/2\beta} < B^{1/2}e^{1/2\beta}$ (which in turn implies $\frac J\beta < \log B + \frac1\beta$), this becomes
\begin{equation*}
\sum_{A\le n\le B} n^{-1/2} \min\{ T, \| \beta\log n \|^{-1} \} \ll \frac{e^{1/2\beta}}{1-e^{-1/2 \beta}} \bigg( TA^{-1/2} + \frac{e^{1/\beta}}{\beta} (\log B +\beta^{-1}) B^{1/2} \bigg),
\end{equation*}
which is the same as the assertion~\eqref{beta gt 0}. When $\beta\ge1$, the simpler assertion~\eqref{beta ge 1} follows the first assertion by bounding the appropriate functions of $\beta$ (although it must be checked a little carefully that the bound is indeed valid when $B$ is very close to~1). 
\end{proof}

\section{Many points in arithmetic progression are not zeros}
\label{main proof}

Having completed all of the preliminaries necessary to proceed to the proofs of the main theorems, in this section we establish Theorem~\ref{main theorem}. As mentioned in Section~\ref{strategy}, to prove Theorem~\ref{main theorem} it suffices to verify the estimates~\eqref{eq:1st moment bound} and~\eqref{eq:2nd moment bound}. These estimates are established below in Propositions~\ref{main thm first moment prop} and~\ref{main thm second moment prop}, respectively.

Let $\chi\mod q$ be a fixed primitive character. Let $\alpha$ and $\beta$ be fixed real numbers, and define $s_k = \tfrac{1}{2} + 2 \pi i (\alpha + k \beta)$. Recall that for $P_1(x)=x$, 
\[
  M_1(s) = M(s,\chi,P_1) = \sum_{m\le X} \frac{\chi(m)\mu(m)}{m^s} \bigg( 1 - \frac{\log m}{\log X} \bigg);
\]
here $X=T^\theta$ where the parameter $0<\theta<1$ will be specified later. We note the trivial bound
\begin{equation}
M_1(\tfrac12+it) \ll \sum_{m\le X} m^{-1/2} \ll X^{1/2}.
\label{M trivial bound}
\end{equation}
We define the corresponding first and second moments
\begin{equation}
   S_1(T,P_1) = \sum_{k=1}^T L(s_k, \chi) M_1(s_k)
\label{S1 def}
\end{equation}
and 
\begin{equation}
   S_2(T,P_1) = \sum_{k=1}^T |L(s_k, \chi) M_1(s_k)|^2.
\label{S2 def}
\end{equation}
The smoothing factor $1-(\log m)/\log X$ in the definition of the mollifier $M_1(s)$ has been introduced so that evaluating the second moment is more convenient, although it slightly complicates the evaluation of the first moment.

In the course of our analysis, we will encounter the coefficients
\begin{equation}
  a_n = \sum_{\substack{\ell m=n \\ \ell\le qU \\ m\le X}} \mu(m) \bigg( 1 - \frac{\log m}{\log X} \bigg),
\label{an definition}
\end{equation}
where $U$ is a parameter to be specified later. The mollifier $M_1(s)$ has been chosen to make the coefficients $a_n$ small; the following lemma gives us more precise information.

\begin{lemma}
We have $a_n = \Lambda(n)/\log X$ for $1<n\le X$, while $|a_n| \le 2^{\omega(n)}$ for $X < n \le qUX$.
\label{an lemma}
\end{lemma}
\begin{proof}
The bound $|a_n| \le 2^{\omega(n)}$ for any $n$ follows from the fact that each term in the sum~\eqref{an definition} defining $a_n$ is at most~1 in size, while the number of nonzero terms is at most the number of squarefree divisors~$m$ of~$n$. Now suppose $1<n\le X$, so that the sum defining $a_n$ includes all of the factorizations of $n$; we can then evaluate exactly
\[
a_n = \sum_{m\mid n} \mu(m) \bigg( 1 - \frac{\log m}{\log X} \bigg) = \sum_{m\mid n} \mu(m) - \frac1{\log X} \sum_{m\mid n} \mu(m)\log m = 0 + \frac1{\log X} \Lambda(n)
\]
using two well-known divisor sums.
\end{proof}

We are now ready to establish the required lower bound for the first moment $S_1(T,P_1)$; in fact we can obtain an asymptotic formula.

\begin{prop}  \label{main thm first moment prop}
Let $T$ be a positive integer, and define $S_1(T,P_1)$ as in equation~\eqref{S1 def}, where $X=T^{\theta}$ with $0 < \theta < 1$. Then $S_1(T,P_1) = T + O_{\chi,\alpha,\beta,\theta}  \big( {T}/{\sqrt{\log T}}\big)$.
\end{prop}

\begin{proof}
In this proof, all constants implicit in the $O$ and $\ll$ notations may depend upon $\chi$ (hence $q$), $\alpha$, $\beta$, and $\theta$. Set $U = 3(|\alpha|+|\beta|)T$, and notice that $U > 2 |\alpha + k \beta| = 2|\mathop{\rm Im} s_k|/2\pi$ for any $1\le k\le T$. Therefore Proposition~\ref {afedirichlet} can be applied (with $\sigma=\frac12$) to yield
\[
   L(s_k, \chi) = \sum_{\ell \le qU} \frac{\chi(\ell)}{\ell^{s_k}} + O ( T^{-1/2} ). 
\]
Using this truncated series and the bound~\eqref{M trivial bound}, we obtain
\begin{align*}
  L(s_k,\chi) M_1(s_k) &= \sum_{\ell\le qU} \frac{\chi(\ell)}{\ell^{s_k}} \sum_{m\le X} \frac{\chi(m)\mu(m)}{m^{s_k}} \bigg( 1 - \frac{\log m}{\log X} \bigg) + O\big( |M_1(s_k)| T^{-1/2}\big) \\
  &= \sum_{n\le qUX} \frac{\chi(n) a_n}{n^ {s_k}} + O\big( X^{1/2} T^{-1/2} \big),
\end{align*}
where $a_n$ is defined in equation~\eqref {an definition}. Therefore, since $a_1=1$,
\begin{align}
  S_1(T,P_1) &= \sum_{k=1}^T \bigg( \sum_{n\le qUX} \frac{\chi(n) a_n}{n^{s_k}} + O\big( X^{1/2} T^{-1/2} \big) \bigg) \notag \\
  &= \sum_{n\le qUX} \frac{\chi(n)a_n}{n^{1/2+2\pi i\alpha}} \sum_{k=1}^T n^{-2\pi i k\beta} + O\big( T^{1/2}X^{1/2} \big) \notag \\
  &= T + O\bigg( \sum_{1<n\le qUX} \frac{|a_n|}{n^{1/2}} \min\{ T, \| \beta\log n \|^{-1} \} + T^{1/2}X^{1/2} \bigg),
  \label{used geometric series}
\end{align}
where we have used the standard geometric series estimate
\begin{equation}
  \label{eq:geometric sum}
  \sum_{k=1}^T n^{-2\pi ik\beta} \ll \min\{ T, \| \beta\log n \|^{-1} \}.
\end{equation}

Thanks to Lemma~\ref{an lemma}, the sum in the error term in equation~\eqref{used geometric series} can be bounded above by
\begin{multline*}
\sum_{1<n\le qUX} \frac{|a_n|}{n^{1/2}} \min\{ T, \| \beta\log n \|^{-1} \} \le \frac1{\log X} \sum_{1<n\le Y} \frac{\Lambda(n)}{n^{1/2}} \min\{ T, \| \beta\log n \|^{-1} \} \\
+ \frac1{\log X} \sum_{Y<n\le X} \frac{\Lambda(n)}{n^{1/2}} \min\{ T, \| \beta\log n \|^{-1} \} + \sum_{X<n\le qUX} \frac{2^{\omega(n)}}{n^{1/2}} \min\{ T, \| \beta\log n \|^{-1} \}
\end{multline*}
for any positive integer $Y \le X$. In the first sum we bound the minimum by $T$ and use the Chebyshev bound $\sum_{n \le x} \Lambda(n) \ll x$; in the second sum we use $\Lambda(n) \le \log X$; and in the final sum we use $2^{\omega(n)} \ll n^\ep$, where $\ep = \frac15\min\{ \theta,1-\theta \}$. The result is
\begin{multline*}
\sum_{1<n\le qUX} \frac{|a_n|}{n^{1/2}} \min\{ T, \| \beta\log n \|^{-1} \} \ll \frac{TY^{1/2}}{\log X} \\
+ \sum_{Y<n\le X} n^{-1/2} \min\{ T, \| \beta\log n \|^{-1} \} + (qUX)^\ep \sum_{X<n\le qUX} n^{-1/2} \min\{ T, \| \beta\log n \|^{-1} \},
\end{multline*}
which in conjunction with Proposition~\ref{one fell swoop prop} yields
\begin{multline*}
\sum_{1<n\le qUX} \frac{|a_n|}{n^{1/2}} \min\{ T, \| \beta\log n \|^{-1} \} \ll \frac{T Y^{1/2}}{\log X} 
+  T Y^{-1/2} + X^{1/2} \log X \\
+ T^{2\ep} \big( T X^{-1/2} + (TX)^{1/2} \log TX \big).
\end{multline*}
We conclude from equation~\eqref {used geometric series} that
\begin{multline*}
S_1(T,P_1) = T + O\bigg( \frac{TY^{1/2}}{\log X} +  T Y^{-1/2} +   X^{1/2} \log T \\
+ T^{2\ep} \big( T X^{-1/2} + (TX)^{1/2} \log T \big) + T^{1/2}  X^{1/2} \bigg).
\end{multline*}
Choosing $Y=\log T$ and simplifying the error terms yields
\begin{align*}
 S_1(T,P_1) & = T + O \bigg( \frac T{\sqrt{\log X}} + T^{2\ep} \big( TX^{-1/2} + T^{1/2}X^{1/2} \log T \big) \bigg) \\
         & = T + O \bigg( \frac T{\sqrt{\log T}} + T^{1-\theta/2+2\ep}+ T^{(\theta+1)/2+2\ep} \log T  \bigg).
\end{align*}
By the choice of $\ep$, the error term is $O\big( T/\sqrt{\log T} \big)$, which establishes the proposition.
\end{proof}

We turn now to the derivation of an upper bound for the corresponding second moment $S_2(T,P_1)$. Rather than parallel the evaluation of the first moment by using a discrete method, we will employ a version of Gallagher's lemma to convert the problem into one of bounding certain integrals. We say that a set $\S$ of real numbers is {\em $\kappa$-well-spaced} if $|u-u'| \ge \kappa$ for any distinct elements $u$ and $u'$ of~$\S$.

\begin{lemma}
Let $F$ and $G$ be differentiable functions, and let $T_1<T_2$ be real numbers. If $\S$ is a $\kappa $-well-spaced set of real numbers contained in $[T_1+ \kappa/2,T_2-\kappa/2]$, then
\begin{multline*}
  \sum_{u\in\S} |F(u)G(u)|^2  \ll_\kappa \int_{T_1}^{T_2} \big| F(t) G(t) \big|^2 \, dt \\
  + 
  \bigg( \int_{T_1}^{T_2} \big| F(t) G(t) \big|^2 \, dt  \bigg)^{1/2}
   \bigg(  \int_{T_1}^{T_2} \big| F(t) G'(t) \big|^2 \, dt  \bigg)^{1/2}  \\
  +  \bigg( \int_{T_1}^{T_2} \big| F(t) G(t) \big|^2 \, dt  \bigg)^{1/2}
    \bigg( \int_{T_1}^{T_2} \big| F'(t) G(t) \big|^2 \, dt \bigg)^{1/2}.
\end{multline*}
\label{product Gallagher lemma}
\end{lemma}

\begin{proof}
Setting $f(t) = H(t)^2$ in~\cite[Lemma 1.2]{Mo} gives
\[
\sum_{u\in\S} |H(u)|^2 \le \frac1 \kappa \int_{T_1}^{T_2}|H(t)|^2 \, dt  + \int_{T_1}^{T_2} |2H(t)H'(t)|  \,dt.
\]
When we set $H(u) = F(u)G(u)$ and use the triangle inequality, we obtain
\[
\sum_{u\in\S} |F(u)G(u)|^2 \ll_\kappa \int_{T_1}^{T_2} |F(t)G(t)|^2 \,dt + \int_{T_1}^{T_2} |F^2(t)G(t)G'(t)| \,dt + \int_{T_1}^{T_2} |F(t)F'(t)G^2(t)| \,dt.
\]
Applying the Cauchy--Schwarz inequality to the last two integrals establishes the proposition.
\end{proof}

Lemma~\ref {product Gallagher lemma} will allow us to reduce the proof of Proposition~\ref {main thm second moment prop} to evaluating integrals of the shape
\[
   \int_{T_1}^{T_2} |L(\tfrac12+iu,\chi) A(\tfrac12+iu)|^2 \,du  \quad\text{and}\quad \int_{T_1}^{T_2} |L'(\tfrac12+iu,\chi) A(\tfrac12+iu)|^2 \,du  
\]
for certain Dirichlet polynomials $A(s)$. Integrals of this type have played an important role in
analytic number theory. For instance, they were introduced by Bohr and Landau~\cite{BL} in their work 
on zero-density estimates. In~\cite{Le}, Levinson employed such integrals in his famous proof that at least one third of the zeros of the Riemann zeta function are simple and lie on the critical line. Recently, Bauer \cite{Ba} studied
integrals of the above type in order to establish analogues of Levinson's theorem to Dirichlet $L$-functions. 
In his work he proves the following result. Recall that $M(s,\chi,P)$ was defined in equation~\eqref{eq:mollifier}.

\begin{lemma}
\label{Bauer lemma}
Let $\chi\mod q$ be a Dirichlet character. Let $Q_1$ and $Q_2$ be polynomials with $Q_1(0)=Q_2(0)=0$, and let $B_1$ and $B_2$ be the sum of the absolute values of the coefficient of $Q_1$ and $Q_2$, respectively. Let $a$ and $b$ be distinct, nonzero complex numbers of modulus at most~$1$. Let $0<\ep<\frac12$ be a real number, let $T$ and $X$ be real numbers satisfying $T^\ep \le X \le T^{1/2-\ep}$, and write $\L = \log ( qT/2\pi )$. There exists a positive real number $\delta=\delta(\ep)$ such that
\begin{multline}
\label{take partials of this}
  \int_1^T L\big( \tfrac12+it + \tfrac a\L, \chi \big) L\big( \tfrac12-it - \tfrac b\L, \bar\chi \big) M\big( \tfrac12+it, \chi, Q_1 \big) M\big( \tfrac12-it, \bar\chi,Q_2 \big) \,dt \\
  {}= T \bigg( Q_1(1)Q_2(1) + \frac{e^{b-a}-1}{b-a} E \bigg) + O_\ep(B_1B_2T^{1-\delta}),
\end{multline}
where
\begin{multline}
\label{E def}
E = \frac{\log T}{\log X} \int_0^1 Q_1'(t) Q_2'(t) \,dt + \int_0^1 \big( bQ_1'(t)Q_2(t) - aQ_1(t)Q_2'(t) \big) \,dt \\
{}- ab\frac{\log X}{\log T} \int_0^1 Q_1(t)Q_2(t) \, dt.
\end{multline}
In particular,
\begin{equation}
\label{Q2 equals Q1}
  \int_1^T \big| L\big( \tfrac12+it, \chi \big) M\big( \tfrac12+it, \chi, Q_1 \big) \big|^2 \,dt = T \bigg( Q_1(1)^2 + \frac{\log T}{\log X} \int_0^1 Q_1'(t)^2 \,dt \bigg) + O_\ep(B_1^2T^{1-\delta}).
\end{equation}
The $O$-constants are uniform in all the parameters as long as $q=o(\log T)$.
\end{lemma}

\begin{proof}
The first assertion is simply a rephrasing of \cite[Theorem 1]{Ba}, with the dependence of the error term on the coefficients of $Q_1$ and $Q_2$ made explicit; the second assertion follows from the first on setting $Q_2=Q_1$ and letting $a$ and $b$ tend to~0.
\end{proof}

\begin{lemma}
\label{Bauer derivative lemma}
With the notation and assumptions of Lemma~\ref{Bauer lemma}, we have
\[
\int_1^T \big| L'\big( \tfrac12+it, \chi \big) M\big( \tfrac12+it, \chi, Q_1 \big) \big|^2 \,dt = \tilde E T\L^2 + O_\ep(B_1^2T^{1-\delta}\L^2),
\]
where
\[
\tilde E = \frac{\log T}{3\log X} \int_0^1 Q_1'(t)^2 \,dt + \int_0^1 Q_1(t)Q_1'(t) \,dt + \frac{\log X}{\log T} \int_0^1 Q_1(t)^2 \, dt.
\]
The $O$-constant is uniform in all the parameters as long as $q=o(\log T)$.
\end{lemma}

\begin{proof}
Define
\begin{multline}
g(a,b) = \int_1^T L\big( \tfrac12+it + \tfrac a\L, \chi \big) L\big( \tfrac12-it - \tfrac b\L, \bar\chi \big) M\big( \tfrac12+it, \chi, Q_1 \big) M\big( \tfrac12-it, \bar\chi,Q_2 \big) \,dt \\
{}- T \bigg( Q_1(1)Q_2(1) + \frac{e^{b-a}-1}{b-a} E \bigg),
\label{gab def}
\end{multline}
where $E$ is defined in equation~\eqref {E def}; Lemma~\ref{Bauer lemma} tells us that $g(a,b) \ll_\ep B_1B_2T^{1-\delta}$ uniformly for $|a|,|b|\le 1$ with $a\ne b$. When we apply $\frac{\partial^2}{\partial a\, \partial b}$ to both sides of equation~\eqref{gab def} and then let $a$ and $b$ tend to zero, we obtain
\begin{multline*}
\frac{\partial^2g}{\partial a\,\partial b}(0,0) = -\frac1{\L^2} \int_1^T L'\big( \tfrac12+it, \chi \big) L'\big( \tfrac12-it, \bar\chi \big) M\big( \tfrac12+it, \chi, Q_1 \big) M\big( \tfrac12-it, \bar\chi,Q_2 \big) \,dt \\
{}+ T \bigg( \frac{\log T}{3 \log X} \int_0^1 Q_1'(t) Q_2'(t) \,dt + \frac12 \int_0^1 \big( Q_1'(t)Q_2(t) + Q_1(t)Q_2'(t) \big) \,dt \\
{}+\frac{\log X}{\log T} \int_0^1 Q_1(t)Q_2(t) \, dt \bigg);
\end{multline*}
upon setting $Q_2=Q_1$, this becomes
\begin{equation*}
\frac{\partial^2g}{\partial a\,\partial b}(0,0) = -\frac1{\L^2} \int_1^T \big| L'\big( \tfrac12+it, \chi \big) M\big( \tfrac12+it, \chi, Q_1 \big) \big|^2 \,dt + \tilde ET.
\end{equation*}
We have thus reduced the proposition to showing that $\frac{\partial^2g}{\partial a\,\partial b}(0,0) \ll_\ep B_1^2 T^{1-\delta}$ (when $Q_2=Q_1$).

From its definition~\eqref {gab def} we see that $g(a,b)$ is holomorphic in each variable. Thus if we set $S_1 = \{z\in \C \colon |z|=\frac13\}$ and $S_2 = \{w\in \C \colon |w|=\frac23\}$, then by Cauchy's integral formula,
\[
\frac{\partial g}{\partial a}(a,b) = \frac{1}{2 \pi i} \int_{S_1} \frac{g(z,b)}{(z-a)^2} \,dz
\]
when $|a|<\frac13$, and hence
\[
\frac{\partial^2 g}{\partial a \,\partial b}(a,b) = \frac{1}{(2 \pi i)^2} \int_{S_2} \int_{S_1} \frac{g(z,w)}{(z-a)^2 (w-b)^2} \,dz \,dw
\]
when $|b|<\frac23$. Setting $a=b=0$ and using the known bound $g(z,w) \ll_\ep B_1B_2T^{1-\delta}$, we obtain $\frac{\partial^2 g}{\partial a \,\partial b}(0,0) \ll_\ep B_1B_2T^{1-\delta}$ as required.
\end{proof}

\begin{prop}
\label{main thm second moment prop}
Let $T$ be a positive integer, and define $S_2(T,P_1)$ as in equation~\eqref{S2 def}, where $X=T^{\theta}$ with $0 < \theta < \frac12$. Then $S_2(T,P_1) \ll_{\chi,\alpha,\beta,\theta} T \log T$.
\end{prop}

\begin{proof}
In this proof, all constants implicit in the $O$ and $\ll$ notations may depend upon $\chi$ (hence $q$), $\alpha$, $\beta$, and $\theta$.
Moreover, we may assume $\beta >0$, since $S_2(T,P_1)= \overline{S_2(T,P_1)}$. 
 If we set $T_1 = 2\pi(\alpha+\beta/2)$ and $T_2 = 2\pi(\alpha+\beta/2+\beta T)$, then the points $\{2\pi(\alpha+k \beta)\colon 1\le k\le T\}$ form a $\kappa$-well-spaced set contained in $[T_1+\kappa/2,T_2-\kappa/2]$, where $\kappa=2\pi\beta$. By Proposition~\ref {product Gallagher lemma} applied to this set with $F(t)=L(\frac12+it)$ and $G(t)=M_1(\frac12+it)$, we obtain
\begin{equation}
\label{used Gallagher}
S_2(T,P_1) \ll I_1 + I_1^{1/2} I_2^{1/2} + I_1^{1/2} I_3^{1/2},
\end{equation}
where we have defined
\begin{align*}
   I_1 &= \int_{T_1}^{T_2} |L(\tfrac{1}{2}+it,\chi)M_1(\tfrac{1}{2}+it)|^2 \,dt = \int_1^{T_2} |L(\tfrac{1}{2}+it,\chi)M_1(\tfrac{1}{2}+it)|^2 \,dt + O(1) \\
   I_2 &= \int_{T_1}^{T_2} |L(\tfrac{1}{2}+it,\chi) M_1'(\tfrac{1}{2}+it)|^2 \,dt = \int_1^{T_2} |L(\tfrac{1}{2}+it,\chi) M_1'(\tfrac{1}{2}+it)|^2 \,dt + O(1) \\
   I_3 &= \int_{T_1}^{T_2} |L'(\tfrac{1}{2}+it,\chi) M_1(\tfrac{1}{2}+it)|^2 \,dt = \int_1^{T_2} |L'(\tfrac{1}{2}+it,\chi) M_1(\tfrac{1}{2}+it)|^2 \,dt + O(1).
\end{align*}
(Actually the definition of $I_2$ should have $\frac d{dt} M_1(\frac12+it)$ instead of $M_1'(\frac12+it)$, but the two expressions differ only by a factor of $i$ which is irrelevant to the magnitude; a similar comment applies to the term $L'(\frac12+it,\chi)$ in~$I_3$.)

The first integral $I_1$ can be evaluated by Lemma~\ref{Bauer lemma}: taking $Q_1(x)=x$ in equation~\eqref{Q2 equals Q1} yields
\[
I_1 = T_2 \bigg( 1 + \frac{\log T_2}{\log X} \bigg) + O\big( T_2^{1-\delta} \big) \ll T.
\]
We next apply equation~\eqref{Q2 equals Q1} with $Q_1(x)=x(x-1)\log X$, so that $M(\frac12+iu,\chi,Q_1) = M_1'(\frac12+iu)$. This application of Lemma~\ref{Bauer lemma} to $I_2$ gives
\[
I_2 = \frac{T_2}3 \log T_2 \cdot \log X + O\big( T_2^{1-\delta}\log^2 X \big) \ll T\log^2 T.
\]
Finally, we apply Lemma~\ref {Bauer derivative lemma} with $Q_1(x)=x$ to evaluate the last integral:
\[
I_3 = \bigg( \frac{\log T_2}{3\log X} + \frac12 + \frac{\log X}{3\log T_2} \bigg) T_2 \log^2 (\tfrac{qT_2}{2 \pi}) + 
O\big( T_2^{1-\delta} \log^2 T_2 \big) \ll T\log^2 T.
\]
Inserting these three estimates into equation~\eqref {used Gallagher} establishes the desired upper bound $S_2(T,P_1) \ll T\log T$ and hence the proposition.
\end{proof}

As was mentioned in Section~\ref {strategy}, the combination of Propositions~\ref {main thm first moment prop} and~\ref {main thm second moment prop} completes the proof of Theorem~\ref{main theorem}. We remark that our evaluation of the expression on the right-hand side of equation~\eqref{used Gallagher} actually produces an asymptotic formula for that expression. Consequently, any attempt to remove the factor of $\log T$ from the upper bound would need to avoid invoking Gallagher's lemma.

\section{Lowest nonzero term }
\label {lowest nonzero proof}

In this final section, we establish Theorem \ref{lowest nonzero theorem} which provides a bound for the least $k$ for which $L(s_k,\chi) \ne 0$. Before we begin we define, for any positive real number $L$, the function
\begin{equation}
\label{w def}
w_L(x) = \exp\bigg( \frac{L\log x}{\log\log x} \bigg).
\end{equation}
It will be convenient to be able to manipulate expressions involving $w_L(x)$ according to the following lemma.

\begin{lemma}
\label{w simplify lemma}
For any positive real numbers $L$, $\alpha$, and $\ep$:
\begin{enumerate}
\item $w_L(x)^\alpha = w_{\alpha L}(x)$;
\item $w_L(x^\alpha) \ll_{L,\alpha,\ep} w_{\alpha L+\ep}(x)$;
\item if $y > x w_L(x)$, then $x \ll_{L,\ep} y/w_{L-\ep}(y)$.
\end{enumerate}
\end{lemma}

\begin{proof}
Part (a) is obvious from the definition. Because $\alpha L/(\log\log x+\log\alpha) < (\alpha L+\ep)/\log\log x$ for $x$ large (even if $\log\alpha$ is negative), we have
\begin{align*}
w_L(x^\alpha) = \exp\bigg( \frac{L \log x^\alpha}{\log\log x^\alpha} \bigg) &= \exp\bigg( \frac{\alpha L \log x}{\log\log x + \log\alpha} \bigg) \\
&\ll_{L,\alpha,\ep} \exp\bigg( \frac{(\alpha L+\ep) \log x}{\log\log x} \bigg) = w_{\alpha L+\ep}(x),
\end{align*}
which establishes part (b). Finally, $y/w_{L-\ep}(y)$ is an eventually increasing function, so
\[
\frac y{w_{L-\ep}(y)} \gg \frac{xw_L(x)}{w_{L-\ep}(xw_L(x))} \gg_{L,\ep} \frac{xw_L(x)}{w_{L-\ep}(x^{1+\ep/L})},
\]
where the latter inequality holds because $w_L(x)$ grows more slowly than any power of~$x$. Applying part (b), with $\ep^2/L$ in place of $\ep$ and $L-\ep$ in place of $L$ and $1+\ep/L$ in place of $\alpha$, gives
\[
\frac y{w_{L-\ep}(y)} \gg_{L,\ep} \frac{xw_L(x)}{w_{(L-\ep)(1+\ep/L)+\ep^2/L}(x)} = x,
\]
which establishes part (c).
\end{proof}

Let $\chi\mod q$ be a fixed non-principal character. Let $0 \le \alpha < \beta$ be fixed real numbers, and define $s_k = \tfrac{1}{2} + 2 \pi i (\alpha + k \beta)$ as before. Recall that for $P_2(x)=1$, 
\[
  M_2(s) = M(s,\chi,P_2) = \sum_{m\le X} \frac{\chi(m)\mu(m)}{m^s};
\]
here $1\le X\le T$ will be specified later. As in the previous section, we note the trivial bound $M_2(\frac12+it) \ll X^{1/2}$. We examine the sum
\[
S_1(T,P_2) = \sum_{k=1}^T L(s_k,\chi) M_2(s_k),
\]
as showing that $S_1(T,P_2) \ne 0$ for a particular $T$ is sufficient to show that there exists a positive integer $k\le T$ for which $L(s_k,\chi)\ne0$. It turns out that the simpler mollifier $M_2$ yields a much better error term for the first moment than the more complicated mollifier $M_1$ from the previous section, as a comparison of the following proposition with Proposition~\ref {main thm first moment prop} will show.


\begin{prop}
\label{asymptotic for second first moment prop}
Suppose that $\beta \ge 1$ and $T\ge\max\{\beta,q\}$. For any constant $L>\log 2$,
\[
S_1(T,P_2) = T + O_L\big( w_L( \beta^{3/2}q^{1/2}T^{3/2} ) \beta^{3/4}q^{1/4}T^{3/4} \log \beta qT \big).
\]
\end{prop}

\begin{proof}
Set $U = 4\beta T$, and notice that $U \ge 2 \beta(T+1) \ge 2 (\alpha + k \beta) = 2|\mathop{\rm Im} s_k|/2\pi$ for any $1\le k\le T$. Therefore Proposition~\ref {afedirichlet} implies that
\[
   L(s_k,\chi) = \sum_{\ell\le qU} \frac{\chi(\ell)}{\ell^{s_k}} + O\big( (\beta T)^{-1/2}\log q \big).
\]
Using the trivial bound for $M_2(s_k)$, we see that
\begin{align*}
L(s_k,\chi) M_2(s_k) &= \sum_{\ell\le qU} \frac{\chi(\ell)}{\ell^{s_k}} \sum_{m\le X} \frac{\chi(m)\mu(m)}{m^s} + O\big( |M_2(s_k)| (\beta T)^{-1/2}\log q \big) \\
&= \sum_{n\le qUX} \frac{\chi(n)b_n}{n^{s_k}}  + O\big( X^{1/2} (\beta T)^{-1/2}\log q \big),
\end{align*}
where the coefficients $b_n$ are defined by
\begin{equation*}
  b_n = \sum_{\substack{\ell m=n \\ \ell\le qU \\ m\le X}} \mu(m).
\end{equation*}

Since $b_1=1$, we have
\begin{align*}
  S_1(T,P_2) &= \sum_{k=1}^T \bigg( \sum_{n\le qUX} \frac{\chi(n) b_n}{n^{s_k}} + O\big( X^{1/2} (\beta T)^{-1/2}\log q \big) \bigg) \\
  &= \sum_{n\le qUX} \frac{\chi(n)b_n}{n^{1/2+2\pi i\alpha}} \sum_{k=1}^T n^{-2\pi i k\beta} + O\big( \beta^{-1/2}T^{1/2}X^{1/2}\log q \big) \\
  &= T + O\bigg( \sum_{1<n\le qUX} \frac{|b_n|}{n^{1/2}} \min\{ T, \| \beta\log n \|^{-1} \} + \beta^{-1/2}T^{1/2}X^{1/2}\log q \bigg),
\end{align*}
by the geometric series estimate~\eqref {eq:geometric sum}. Just as in the proof of Lemma~\ref {an lemma}, it is easy to show that $|b_n| \le 2^{\omega(n)}$ for any $n$. Moreover, if $1<n\le X$ then the sum defining $b_n$ includes all of the factorizations of $n$, and so $b_n=0$ for these values of~$n$. We conclude that
\begin{equation}
\label{reduced to min sum}
S_1(T,P_2) = T + O_L\bigg( \sum_{X<n\le qUX} \frac{2^{\omega(n)}} {\sqrt n} \min\{ T, \| \beta\log n \|^{-1} \} + \beta^{-1/2}T^{1/2} X^{1/2} \log q \bigg).
\end{equation}
From the inequality $\omega(n) < (1+o(1)) (\log n)/\log\log n$~\cite[Theorem 2.10]{MV}, we see that $2^{\omega(n)} = \exp\big( \omega(n)\log 2 \big) \ll_L w_L(n)$, where $w_L$ was defined in equation~\eqref {w def}. Therefore
\begin{equation*}
S_1(T,P_2) = T + O_L\bigg( w_L(qUX) \sum_{X<n\le qUX} \frac{\min\{ T, \| \beta\log n \|^{-1} \}} {\sqrt n} + \beta^{-1/2}T^{1/2} X^{1/2} \log q \bigg).
\end{equation*}
By Proposition~\ref{one fell swoop prop} (recalling that we are assuming $\beta\ge1$),
\begin{align*}
S_1(T,P_2) &= T + O_L\big( w_L(qUX) \big( \beta TX^{-1/2} + (qUX)^{1/2} \log qUX \big) + \beta^{-1/2}T^{1/2} X^{1/2} \log q \big) \\
&= T + O_L\big( w_L(\beta qTX) ( \beta TX^{-1/2} + (\beta qTX)^{1/2} \log \beta qTX) \big).
\end{align*}
We now set $X=\sqrt{\beta T/q}$, which is in the required range $1\le X\le T$ thanks to the hypothesis $T\ge\max\{\beta,q\}$. We then obtain
\begin{equation*}
S_1(T,P_2) = T + O_L\big( w_L( \beta^{3/2}q^{1/2}T^{3/2} ) \beta^{3/4}q^{1/4}T^{3/4} \log \beta qT \big)
\end{equation*}
as claimed.
\end{proof}

\begin{proof}[Proof of Theorem \ref{lowest nonzero theorem}]
We begin under the assumption $b\ge2\pi$, so that $\beta = b/{2\pi} \ge 1$. Define $L = \frac12(\frac D8+\log 2)$ and $\ep = \frac12(\frac D8-\log 2)$, so that $L>\log 2$ and $\ep>0$. Suppose that
\[
T > \beta^3q \cdot w_D(\beta^3q) = \beta^3q \cdot w_{8(L+\ep)}(\beta^3q).
\]
It follows from Lemma~\ref {w simplify lemma}(c) that $\beta^3q < T/w_{8L+7\ep}(T)$. Therefore Proposition~\ref {asymptotic for second first moment prop} implies
\begin{align*}
S_1(T,P_2) &= T + O_L\bigg( w_L( T^2 ) \bigg( \frac T{w_{8L+7\ep}(T)} \bigg)^{1/4} T^{3/4} \log T^2 \bigg) \\
&= T + O_{L,\ep}\bigg( T \frac{w_{2L+\ep}(T)\log T}{w_{2L+7\ep/4}(T)} \bigg) = T + O_{L,\ep}\bigg( \frac{T \log T}{w_{3\ep/4}(T)} \bigg)
\end{align*}
by Lemma~\ref {w simplify lemma}(b) and (a). Since the error term is $o(T)$, the first moment $S_1(T,P_2)$ is nonzero when $T$ is large enough in terms of $L$ and $\ep$. In other words, $S_1(T,P_2)\ne0$ when $T \gg_D \beta^3q \cdot w_D(\beta^3q)$, which implies Theorem~\ref {lowest nonzero theorem}(a).

Now suppose that $0<b<2\pi$, so that $0<\beta<1$. Note that Theorem~\ref {lowest nonzero theorem}(a) implies that if $2\pi\le c \le 4\pi$, then there exists a positive integer $j \ll_D q \exp\big( D (\log q)/\log\log q \big)$ such that $L(\frac12+i(a+cj),\chi)\ne0$, where the implicit constant is independent of~$c$. Applying this remark with $c=\lceil 2\pi/b \rceil b$ which is between $2\pi$ and $4\pi$, we see that there exists a positive integer $j \ll_D q \exp\big( D (\log q)/\log\log q \big)$ such that $L(\frac12+i(a+b\lceil 2\pi/b \rceil j),\chi)\ne0$. In other words, there exists a positive integer $k \ll_D \lceil 2\pi/b \rceil q \exp\big( D (\log q)/\log\log q \big)$ such that $L(s_k,\chi)\ne0$, which establishes Theorem~\ref {lowest nonzero theorem}(b).
\end{proof}

\noindent {\bf Acknowledgements}
The authors thank the Banff International Research Station for their hospitality in May 2009 for a Research in Teams meeting. We greatly appreciated the excellent working conditions and the beautiful scenery of the Canadian Rockies. We also thank Maksym Radziwill for several interesting conversations and for informing us of the relevance of~\cite{FSZ} to our work.


\begin{thebibliography}{99}

\bibitem{BM}
R. Balasubramanian and V. K. Murty, {\em Zeros of Dirichlet $L$-functions}, Ann. Sci. �cole Norm. Sup. (4) {\bf25} (1992), no. 5, 567--615.

\bibitem{Betal}
P.T. Bateman, J.W. Brown, R.S. Hall, K.E. Kloss, and R.M. Stemmler, {\em Linear relations connecting the imaginary parts of the zeros of the zeta function}, Computers in number theory (Proc. Sci. Res. Council Atlas Sympos. No. 2, Oxford, 1969), Academic Press, London, 1971, pp. 11--19.

\bibitem{Ba}
P. J. Bauer, {\em Zeros of Dirichlet $L$-series on the critical line},  Acta Arith. 93  (2000),  no. 1, 37--52.

\bibitem{BL}
H. Bohr and E. Landau, {\em Sur les z\'{e}ros de la fonction $\zeta(s)$ de Riemann}, Comptes Rendus de l'Acad. des Sciences
(Paris) {\bf 158} (1914), 106--110. 

\bibitem{C}
J.B. Conrey, {\em More than two fifths of the zeros of the Riemann zeta function are on the critical line},  
J. Reine Angew. Math.  399  (1989), 1--26.

\bibitem{CIS}
J.B. Conrey, H. Iwaniec, and K. Soundararajan, {\em Critical zeros of Dirichlet $L$-functions}, Preprint, 2011, 
http://arxiv.org/pdf/0912.4908.

\bibitem{D}
H. Davenport, {\em Multiplicative number theory, Third Edition}, Graduate Texts in Mathematics 74,  Springer--Verlag, New York, 2000. 

\bibitem{FSZ}
K. Ford, K. Soundararajan, and A. Zaharescu, {\em On the distribution of imaginary parts of zeros of the Riemann zeta function, II}, Math. Annalen 343 (2009), 487--505.

\bibitem{GM}
A. Granville and G. Martin, {\em Prime number races}, Amer. Math. Monthly  {\bf 113} (2006), no. 1, 1--33.

\bibitem{Ho}
C. Hooley, {\em On the Barban--Davenport--Halberstam Theorem: VII}, J. London Math. Soc. (2) {\bf 16} (1977), 1--8. 


\bibitem{In2}
A.E. Ingham, {\em  On two conjectures in the theory of numbers},  Amer. J. Math.  64,  (1942). 313--319.

\bibitem{IK} H. Iwaniec and E. Kowalski, {\em Analytic Number Theory}, American Mathematical Society Colloquium
Publications, 53. American Mathematical Society, Providence, RI, 2004.

\bibitem{IS}
H. Iwaniec and P. Sarnak, {\em Dirichlet $L$-functions at the central point}, Number theory in progress, Vol. 2 (Zakopane--Ko\'scielisko, 1997), 941--952, de Gruyter, Berlin, 1999.

\bibitem{Ju}
M. Jutila, {\em On the mean value of $L(\frac12,\chi)$ for real characters}, Analysis {\bf1} (1981), no. 2, 149--161.

\bibitem{LvF}
M.L. Lapidus and  M. van Frankenhuysen, {\em Fractal geometry and number theory. Complex dimensions of fractal strings and zeros of zeta functions}, Birkh\"auser Boston, Inc., Boston, MA, 2000. xii+268 pp. 

\bibitem{Le}
N. Levinson, {\em More than one third of zeros of Riemann's zeta-function are on $\sigma =1/2$},
Advances in Math.  13  (1974), 383--436.

\bibitem{LR}
X. Li and M. Radziwi\l\l, {\em The Riemann zeta function on vertical arithmetic progressions}, preprint (arXiv:1208.2684).

\bibitem{Li}
U.V. Linnik, {\em On the density of the zeros of $L$-series. (Russian. English summary)},
Bull. Acad. Sci. URSS. S\'{e}r. Math. [Izvestia Akad. Nauk SSSR] {\bf 10}, (1946), 35--46. 

\bibitem{Mon}
W.R. Monach, {\em Numerical investigation of several problems in number theory}, Ph. D. Thesis, University
of Michigan (1980), 171 pp.

\bibitem{Mo2}
H.L. Montgomery, {\em The pair correlation of zeros of the zeta function},
Proc. Symp. Pure Math. 24, A.M.S, Providence 1973, 181--193.

\bibitem{Mo}
H. L. Montgomery, {\em Topics in multiplicative number theory}, Lecture Notes in Mathematics, Vol. 227. Springer--Verlag, Berlin--New York, 1971. 

\bibitem{Mo3}
H.L. Montgomery, {\em The zeta function and prime numbers},
Proceedings of the Queen's Number Theory Conference, 1979,
Queen's Univ., Kingston, Ont., 1980, 1--31.

\bibitem{MV}
H.L. Montgomery and R.C. Vaughan, \emph{Multiplicative number theory.
  {I}. {C}lassical theory}, Cambridge Studies in Advanced Mathematics, vol.~97,
  Cambridge University Press, Cambridge, 2007.

\bibitem{Od}
A.M. Odlyzko, {\em On the distribution of spacings between zeros of the zeta function},  
Math. Comp.  48  (1987),  no. 177, 273-308.

\bibitem{Oz}
A. E. \"Ozl\"uk, {\em On the $q$-analogue of the pair correlation conjecture}, J. Number Theory 59 (1996), no. 2, 319--351.


\bibitem{Pu1}
C.R. Putnam, {\em On the non-periodicity of the zeros of the Riemann zeta-function}, 
Amer. J. Math. 76 (1954) 97-99.  

\bibitem{Pu2}
C.R. Putnam, {\em Remarks on periodic sequences and the Riemann zeta-function}
, Amer. J. Math. 76 (1954) 828-830.  

\bibitem{Ra}
V.V. Rane, {\em On an approximate functional equation for Dirichlet $L$-series},  Math. Ann.  264  (1983),  no. 2, 137-145. 

\bibitem{RS} 
M. Rubinstein and P. Sarnak,
{\em Chebyshev's bias}, 
Experiment. Math. 3 (1994), no. 3, 173-197. 

\bibitem{otherRS}
Z. Rudnick and P. Sarnak, {\em Zeros of principal $L$-functions and random matrix theory}, A celebration of John F. Nash, Jr, Duke Math. J. 81 (1996), no. 2, 269--322.

\bibitem{Ru}
R. Rumely, {\em Numerical computations concerning the ERH}, Math. Comp.  61  (1993),  no. 203, 415--440, S17--S23.

\bibitem{Se1}
A. Selberg, {\em Contributions to the theory of Dirichlet's $L$-functions},  Skr. Norske Vid. Akad. Oslo. I.  1946,  (1946). no. 3, 62 pp.

\bibitem{Se2}
A. Selberg,  {\em Contributions to the theory of the Riemann zeta-function},  Arch. Math. Naturvid.  48,  (1946) no. 5, 89�155.

\bibitem{Silb}
L. Silberman, personal communication, 2010.

\bibitem{So}
K. Soundararajan, {\em Nonvanishing of quadratic Dirichlet $L$-functions at $s=\frac12$}, Ann. of Math. (2) {\bf152} (2000), no. 2, 447-488.

\bibitem{T}
E. C. Titchmarsh, \emph{The Theory of the Riemann Zeta-function}, Oxford University
Press, 1986.

\bibitem{vtw}
J. van de Lune, H.J.J. te Riele, and D.T. Winter, {\em On the zeros of the Riemann zeta function in the critical strip. IV},  Math. Comp.  46  (1986),  no. 174, 667--681.

\bibitem{vF}
M. van Frankenhuijsen, {\em Arithmetic progressions of zeros of the Riemann zeta function},  J. Number Theory  115  (2005),  no. 2, 360--370.

\bibitem{Wa1}  M. Watkins, {\em Arithmetic progressions of zeros of Dirichlet $L$-functions}, preprint.

\bibitem{Wi2a}
A. Wintner, {\em Asymptotic distributions and infinite convolutions}, Notes distributed the Institute for Advanced Study (Princeton), 
1938. 

\bibitem{Wi1} 
A. Wintner, {\em On the asymptotic distribution of the remainder term of the prime number theorem},
 Amer. J. Math. 57 (1935), no. 3, 534-538. 

\bibitem{Wi2}
A. Wintner, {\em On the distribution function of the remainder term of the prime number theorem},
 Amer. J. Math. 63 (1941), no. 2, 233-248. 

\end{thebibliography}
\end{document}